\documentclass[reqno]{amsart}%
\usepackage{enumerate}
\usepackage{amsfonts}
\usepackage{amsmath}
\usepackage{amssymb}
\usepackage{cases}
\usepackage{hyperref}
\usepackage{float}
\usepackage{graphicx}%
\newtheorem{theorem}{Theorem}[section]
\newtheorem{lemma}[theorem]{Lemma}
\newtheorem{proposition}[theorem]{Proposition}

\theoremstyle{definition}

\newtheorem{remark}[theorem]{Remark}

\numberwithin{equation}{section}
\def\fin { \vskip 0pt \hfill $\diamond$ \vskip 12pt}
\hypersetup{
colorlinks,
citecolor=blue,
filecolor=black,
linkcolor=blue,
urlcolor=magenta
}
\hypersetup{linktocpage}
\begin{document}
\title[Generalized Boussinesq equation on hyperbolic spaces]{Dispersive estimates and generalized Boussinesq equation on hyperbolic spaces with rough initial data}
\author[L.C.F. Ferreira \ ]{Lucas C. F. Ferreira \ \ }
\address{Lucas C. F. Ferreira \hfill\break University of Campinas, IMECC-Department of
Mathematics, Rua S\'{e}rgio Buarque de Holanda, 651, CEP 13083-859,
Campinas-SP, Brazil}
\email{lcff@ime.unicamp.br}
\author[\ P.T. Xuan]{\ \ Pham T. Xuan}
\address{Pham Truong Xuan \hfill\break Thang Long Institute of Mathematics and Applied Sciences (TIMAS), Thang Long University \hfill\break
Nghiem Xuan Yem, Hoang Mai, Hanoi, Vietnam}
\email{xuanpt@thanglong.edu.vn or phamtruongxuan.k5@gmail.com}
\thanks{LCFF was supported by CNPq (grant: 308799/2019-4 and grant: 312484/2023-2), Brazil.}

\begin{abstract}
We consider the generalized Boussinesq (GBq) equation on the real hyperbolic space $\mathbb{H}^{n}$ ($n\geq2$) in a rough framework based on Lorentz spaces. First, we establish dispersive estimates for the GBq-prototype group, which is associated with a core term of the linear part of the GBq equation, through a manifold-intrinsic Fourier analysis and estimates for oscillatory integrals in $\mathbb{H}^{n}$. Then, we obtain dispersive estimates for the GBq-prototype and Boussinesq groups on Lorentz spaces in the context of $\mathbb{H}^{n}$. Employing those estimates, we obtain local and global well-posedness results and scattering properties in such framework. Moreover, we prove the polynomial stability of mild solutions and leverage this to improve the scattering decay.
\vspace{0.2cm}

 \vspace{0.2cm}


\end{abstract}
\maketitle

\tableofcontents

\font\nho=cmr10





\section{Introduction}

In this work, we analyze the generalized Boussinesq (GBq) equation on the real hyperbolic space $\mathbb{H}^{n}$ ($n\geq2$), which reads as
\begin{equation}%
\begin{cases}
u_{tt}-\Delta_{x}u+\Delta_{x}^{2}u+\Delta_{x}f(u)=0,\, & x\in\mathbb{H}%
^{n},\,t\in\mathbb{R},\\
u(0,x)=u_{0}(x),\, & x\in\mathbb{H}^{n}\\
u_{t}(0,x)=\phi(x)=\Delta_{x}v_{0}(x),\, & x\in\mathbb{H}^{n},
\end{cases}
\label{Bou}%
\end{equation}
where the unknown $u$ is a real function on $(x,t)\in\mathbb{H}^{n}%
\times\mathbb{R}$, the operator $\Delta_{x}$ denotes the Laplace-Beltrami
operator related to the hyperbolic metric, and the nonlinearity $f(u)$ meets
the conditions
\begin{equation}
f(0)=0\text{ and }|f(a_{1})-f(a_{2})|\leq C_{f}(|a_{1}|^{b-1}+|a_{2}%
|^{b-1})|a_{1}-a_{2}|,\text{ for }b>1, \label{Bou-Cond-f}%
\end{equation}
where the constant $C_{f}$ is independent of $a_{1},\,a_{2}\in\mathbb{R}$. The
initial values $u_{0},\,v_{0}:\mathbb{H}^{n}\rightarrow\mathbb{R}$ are given functions.

The GBq equation \eqref{Bou} was introduced by Boussinesq \cite{Bou1872} and
then this equation and its generalisations on the Euclidean space
$\mathbb{R}^{n}$ have been studied by several authors. In what follows, we
first recall concisely some results of the literature for them on the
Euclidean setting. The local well-posedness of the GBq equation on
$\mathbb{R}$ ($n=1$) were proved by Bona and Sachs \cite{Bo1988} by
considering $f\in C^{\infty}(\mathbb{R}),$ $u_{0}\in H^{s+2}(\mathbb{R})$ and
$v_{0}\in H^{s+1}(\mathbb{R})$ with $s>1/2.$ In addition, assuming $s\geq1,$
$1<b<5,$ and that the initial data $[u_{0},v_{0}]$ is close enough to a stable
solitary wave, they showed that the solution can be extended globally.
Tsutsumi and Matahashi \cite{Tsu1991} obtained a similar local well-posedness
result by considering $f(u)=\left\vert u\right\vert ^{b-1}u,$ $b>1,$ $u_{0}\in
H^{1}(\mathbb{R})$ and $v_{0}\in H^{1}(\mathbb{R}).$ Linares
\cite{Linears1993} showed the local well-posedness by assuming either $1<b<5$
and $[u_{0},(v_{0})_{x}]\in L^{2}(\mathbb{R})\times\dot{H}^{-1}(\mathbb{R}),$
or $b>1$ and $[u_{0},(v_{0})_{x}]\in H^{1}(\mathbb{R})\times L^{2}%
(\mathbb{R}).$ Moreover, by using energy conservation law, he obtained that
the solutions can be extended globally in $H^{1}(\mathbb{R})$ for small enough
data $[u_{0},(v_{0})_{x}]$ in $H^{1}(\mathbb{R})\times L^{2}(\mathbb{R})$.
Supposing further regularity on the initial data, Linares and Scialom
\cite{Linears1995} showed results on time-decay and scattering for the small
global solutions obtained in \cite{Linears1993}. The local results in
\cite{Linears1993} were extended to high dimensions by Farah \cite{Fa09}. By
employing the Bourgain Fourier restriction norm approach, and considering
$f(u)=u^{2}$ and $s>-1/4,$ Farah \cite{Fa2009} proved a local well-posedness
result with data $[u_{0},\phi]=[u_{0},(v_{0})_{x}]\in H^{s}(\mathbb{R})\times
H^{s-1}(\mathbb{R})$. The instability of solitary waves in $H^{1}(\mathbb{R})\times L^{2}%
(\mathbb{R})$ and blow-up results for the GBq equation were provided by Liu
\cite{Liu1993,Liu1995}. The blow-up result for an improved Boussinesq type
equation in $W^{2,\infty}([0,\,T],\,H^{2}(0,\,1))$ was studied by Yang and
Wang \cite{Yang2003}. The small amplitude solutions and scattering for the GBq
equation and an improved modified Boussinesq equation on $\mathbb{R}^{n}\text{
(for }n\geq1$) with initial data $[u_{0},\phi]$ in a functional setting of
Besov type were investigated by Cho and Ozawa \cite{Ozawa2007}. In turn, results on
decay and scattering for the GBq equation with initial data $[u_{0},\phi]\in
H^{1}(\mathbb{R})\times L^{2}(\mathbb{R})$ were established by Liu
\cite{Liu1997}. Recently, Liu and Wang \cite{LiuWang} studied dispersive and
dissipative estimates for a dissipative-dispersive linear semigroup and gave
an application to the global well-posedness of the GBq equation with initial
data in $L^{1}(\mathbb{R}^{n})\cap H^{s}(\mathbb{R}^{n})\times|\nabla
|L^{1}(\mathbb{R}^{n})\cap|\nabla|H^{s}(\mathbb{R}^{n})$ for all $n\geq1$ and
a suitable value of $s$. Moreover, Munoz \textit{et al.} \cite{Mu2018} proved
that small solutions of the GBq equation in the energy space for one-dimension
(i.e., $H^{1}(\mathbb{R})\times L^{2}(\mathbb{R})$) must decay to zero as $t\rightarrow\infty$, strongly on slightly proper subsets of the space-time light cone. In addition, Chen \textit{et al.} \cite{Che2023} obtained local
well-posedness, finite-time blow up and small initial data scattering for GBq
equation in the energy space $H^{1}(\mathbb{R}^{n})\times L^{2}(\mathbb{R}%
^{n})$, as well as a large radial initial data scattering for the defocusing
case with $n\geqslant3$. Results on well-posedness and scattering for GBq
equation with initial data in a setting of infinite $L^{2}$-mass based on
weak-$L^{p}$ spaces on $\mathbb{R}^{n}$ ($n\geq1$) were obtained by Ferreira
\cite{Fe2011}. Moreover, we would like to refer some related works in
\cite{Es2012,Es12,Ki2013,Mc1981}.

The purpose of this paper is to establish results on the local and global
well-posedness, asymptotic stability and scattering for the GBq equation
\eqref{Bou} on hyperbolic space $\mathbb{H}^{n}$ with initial data in a class
of infinite energy spaces, namely $L^{(p,\infty)}(\mathbb{H}^{n})$-spaces (aka weak-$L^{p}(\mathbb{H}^{n})$, see Theorems \ref{LWP} and \ref{GWP}). The weak-$L^{p}$ norms of the global mild solutions present a polynomial-time decay as $|t|$
tends to infinity. Under further conditions, the regularity of the mild
solutions in the Lorentz space $L^{(p,d)}(\mathbb{H}^{n})$ is also studied in
item $(ii)$ of those theorems. In addition, we prove a scattering result for
the global mild solutions in such rough framework (see Theorem
\ref{scat}) and an asymptotic stability result (see Theorem \ref{Stable}). Using
this stability result, we are able to improve the scattering result (see
Remark \ref{rem}).

Nonlinear Schr\"{o}dinger, wave and Boussinesq-type equations are fundamental models
in mathematical physics describing nonlinear wave propagation. These equations
often exhibit nonlinear terms that pose challenges in the classical $L^{2}%
$-framework due to issues such as blow-up or lack of global regularity. To
overcome these challenges, Cazenave \textit{et al.} \cite{Ca2001} were the
first to explore the use of $L^{(p,\infty)}$-spaces in the context of Schr\"{o}dinger equations in $\mathbb{R}^{n}$ (Euclidean
case), where the integrability conditions on the solutions are relaxed. This
framework allows for a broader class of solutions and helps in establishing
well-posedness results (Hadamard sense), which ensure the existence,
uniqueness, and continuous dependence on initial data of solutions to the
equations under consideration. More precisely, they addressed the
well-posedness and scattering phenomena in $L^{(p,\infty)}$-spaces by means of
Strichartz estimates and using the mixed space-time setting $L^{(p,\infty
)}(\mathbb{R}\times\mathbb{R}^{n})$ for $p=\frac{(b-1)(N+2)}{2}$. Later, by
using dispersive-type estimates and considering time-polynomial weighted
spaces based on the $L^{(p,\infty)}(\mathbb{R}^{n})$ with $p=b+1$, Ferreira
\textit{et al.} \cite{Fe2009} obtained results on global well-posedness and
asymptotic behavior of solutions, extending previous achievements in the
$L^{p}$-setting due to Cazenave and Weissler \cite{Ca1998, Ca2000}. Moreover,
well-posedness and scattering results for Boussinesq and wave equations have
been obtained via that kind of time-weighted spaces in \cite{Fe2011} and
\cite{Fe2017,Liu2009}, respectively. In these last three works the authors
employed $L^{(p,d)}$-$L^{(p^{\prime},d)}$-dispersive estimates in
$\mathbb{R}^{n},$ where $L^{(p,d)}$ represents the so-called Lorentz space,
$p^{\prime}$ is the conjugate of $p$, and $1\leq d\leq\infty$, and regarded
suitable Kato-type classes as environment for seeking solutions. Recently, extensions of well-posedness and scattering in weak-$L^{p}$ spaces in the framework of real hyperbolic spaces have been obtained by Ferreira and Xuan
\cite{FeXuan2023} for the wave equations. In this work, the authors have used
the dispersive estimates for the wave group obtained by Tataru \cite{Ta2001}
to construct a time-weighted space which guarantees the global well-posedness
of wave equations. Note that we have exponential decays for stability and
scattering properties in the case of wave-type equations.

The dispersive equations on hyperbolic spaces and noncompact symmetric spaces
have been extensively studied in the last two decades. We refer some previous
works (and many references therein) such as
\cite{Anker2012,Anker2014,Anker2015,Ha2011,
MeTa2011,MeTa2012,Ta2001,Zhang2020,Zhang2021} for wave and Klein-Gordon
equations, and \cite{Anker2009,Anker2011,Ba2007,Ba07,Ba2008,Ba2009,Ba2015,
CaHo,Ionescu2000,Ionescu2009,Pi2006} for Schr\"{o}dinger equations. In these works,
the authors used Fourier analysis on hyperbolic spaces (or more general, on
noncompact symmetric spaces) (see \cite{Anker1996,Hel1965,Ionescu2000} for the
theory) and oscillatory integral estimates (see for example \cite{Ke1991,St1986}) to established the dispersive estimates and (radial
weighted-) Strichartz estimates for wave and Schr\"{o}dinger groups. Then, they
used these estimates to obtain the global well-posedness and scattering for
wave and Schr\"{o}dinger equations with initial data in energy spaces. Sofar, to the best of our knowledge, there are no works on the GBq equation \eqref{Bou} in hyperbolic spaces. The main obstacle comes from the higher order Laplace operator appearing in the equation. From this we have a linearized modified GBq equation associated with \eqref{Bou} (see Subsection \ref{S22}) and it is not clear how dispersive estimates would work for the corresponding group, namely the GBq-prototype group. Another important issue is how we can employ the obtained estimates to achieve results on well-posedness and scattering for \eqref{Bou}, covering rough initial data outside the usual energy space $H^{1}%
(\mathbb{H}^{n})\times L^{2}(\mathbb{H}^{n})$.

We tackle the above issues by first proving the $L^{p}$-dispersive estimates for the GBq-prototype group on hyperbolic space
$\mathbb{H}^{n}$. In particular, we develop the methods in
\cite{Anker2014,CaHo,Ionescu2000, Ionescu2009} based on the Fourier analysis
on the hyperbolic spaces and oscillatory integral estimates (via the van der
Corput lemma) to establish these estimates (see Subsection \ref{S31}). Then,
we use the interpolation inequalities to obtain the $L^{(p,d)}$-$L^{(p^{\prime
},d)}$-dispersive estimates, where the decay rate is $\left\vert t\right\vert
^{-\frac{3}{2}}$ for large $\left\vert t\right\vert $ and $\left\vert
t\right\vert ^{-\frac{n}{2}\left(  1-\frac{2}{p}\right)  }$ for small
$\left\vert t\right\vert $ (see Subsection \ref{S32}). By using these
estimates, we prove local and global well-posedness results in the mixed time
weighted spaces $\mathcal{L}_{\beta}^{T}$ and $\mathcal{L}_{\alpha_{1}%
,\alpha_{2}}$, respectively, which are constructed on the basis of
$L^{(p,\infty)}$-spaces (see Section \ref{S4}). The asymptotic behavior
properties (scattering and stability) are established in Section \ref{S5}, where the scattering data also belongs to the space containing rough elements outside $H^{1}(\mathbb{H}^{n})\times L^{2}(\mathbb{H}^{n})$.

In comparison to the Euclidean setting (see \cite{Ca1998, Ca2000, Fe2009,
Fe2011}), we obtain a different polynomial decay of scattering and asymptotic
stability, since the Boussinesq group on $\mathbb{H}^{n}$ presents a different
decay in the dispersive estimates for long times. More precisely, the
$L^{b+1}$ and $L^{(b+1,\infty)}$-norms of the solutions for GBq and
nonlinear Schr\"{o}dinger equations on $\mathbb{R}^{n}$ decay like $t^{-\alpha}$ with
$\alpha=\frac{1}{b-1}-\frac{n}{2(b+1)}<\frac{1}{b}$ as $\left\vert
t\right\vert \rightarrow\infty$, while here we have a faster polynomial decay
$t^{-\frac{3}{2}}$ as $\left\vert t\right\vert \rightarrow\infty$, i.e. with a
higher decay rate $\alpha=\frac{3}{2}$. Another interesting aspect of
comparison to our results is the drastic difference with the wave and
Klein-Gordon equations on $\mathbb{H}^{n}$ where the solutions present
stronger decays of exponential type (see \cite{Ta2001,FeXuan2023}), as
aforementioned. These contrasting differences stem from how the manifold
metric influences the asymptotic behavior of the GBq-prototype and wave
groups, as examined through a Fourier analysis intrinsic to the manifold in question. Our results together with \cite{Fe2011} provide a combined understanding about GBq equation \eqref{Bou} in the frameworks of Euclidean
and non-Euclidean spaces. Let us pointed out that the present paper is the
first work to mention about the well-posedness and asymptotic behavior for
Boussinesq-type equations on non-Euclidean settings.

The outline of this manuscript is as follows. In Section \ref{S21} we provide
some preliminaries about hyperbolic spaces and Fourier analysis on them.
Section \ref{S22} is devoted to presenting an integral formulation for the
problem as well as some associated groups of operators. The subject of Section
\ref{S3} is the key dispersive estimates for the GBq-prototype group. In Sections \ref{S4} and \ref{S5}, we obtain the needed estimates for the Boussinesq group and prove our results on local and global well-posedness
and asymptotic behavior properties for the GBq equation. In Appendix, we give a detailed proof for an estimate of oscillatory integrals used in previous sections.

\

{\bf Notation and further comments:}

\begin{itemize}
\item[$\bullet$] For two real functions $f(x)$ and $g(x)$, we say $f\lesssim
g$ if there exists a universal constant $C>0$ such that $f(x)\leq Cg(x),$ for
all $x$. Also, we write $f\simeq g$ if both $f\lesssim g$ and $g\lesssim f$
holds true.

\item[$\bullet$] Using the kernel and dispersive estimates for GBq-prototype group obtained in this paper (see Lemma \ref{KerEst1} and Theorem \ref{Dispersive}), one can prove the Strichartz estimates and establish the well-posedness and scattering results in the energy space $H^1(\mathbb{H}^n)\times L^2(\mathbb{H}^n)$ for the GBq equation (see also Remark \ref{Strichatz}).

\item[$\bullet$] It seems that the results obtained in this paper can be extended on generalized symmetric non-compact Riemannian manifolds and we hope to treat this problem in a future paper.
\end{itemize}

\section{Boussinesq equations on hyperbolic space}\label{S2}
\subsection{Hyperbolic space and Fourier analysis}\label{S21}
We recall the notion of real hyperbolic space and the Fourier transform. For more details, we refer readers to refs. \cite{Anker1996,Ba2007,Can1997,Hel1965,
Ionescu2000,Ionescu2009}.

Let $(\mathbb{H}^{n},g):=(\mathbb{H}^{n}(\mathbb{R}),g)$ denote a real
hyperbolic space of dimension $n\geq2$ with metric $g$. This hyperbolic space
can be represented by means of the hyperboloid model in $\mathbb{R}^{n+1}$ by
considering the upper sheet of the hyperboloid
\[
\left\{  (x_{0},x_{1},...,x_{n})\in\mathbb{R}^{n+1};\text{ }\,x_{0}\geq1\text{
and }x_{0}^{2}-x_{1}^{2}-x_{2}^{2}...-x_{n}^{2}=1\,\right\}
\]
equipped with the metric
\begin{equation}
g=-dx_{0}^{2}+dx_{1}^{2}+...+dx_{n}^{2}.\label{metric1}
\end{equation}
Following this metric, we define the scalar product $\left< x,y\right> = -x_0y_0 + x_1y_1 + ...+x_ny_n$ for two points $x=(x_0,x_1,...,x_n),\, y=(y_0,y_1,...,y_n)\in \mathbb{H}^n$. The geodesic distance between these points is given by the formula (see \cite[Section 2.1]{Ba2007})
\begin{equation}\label{distance}
d(x,y) = \cosh^{-1}(-\left< x,y\right>).
\end{equation}
In terms of Lie group, the
hyperbolic space $\mathbb{H}^{n}$ can be identified with the homogeneous space
$SO(n,1)/SO(n)$. Here, $SO(n,1)$ stands for the connected Lie group of
$(n+1)\times(n+1)$-matrices $X$ with $\det X=1$ and $X_{00}>0$ that preserve
the scalar product $\left\langle \cdot,\cdot\right\rangle ,$ while $SO(n)$ is
the subgroup of $SO(n,1)$ that keeps the origin $O=(1,0...0)\in\mathbb{H}^{n}$ invariant.

In geodesic polar coordinates, the hyperbolic space $(\mathbb{H}^{n},g)$ can
be represented as
\[
\mathbb{H}^{n}=\left\{  (\cosh r,\omega\sinh r),\,r\geq0,\omega\in
\mathbb{S}^{n-1}\right\}  ,
\]
where $r$ is the radial coordinate and $\omega$ is the angular coordinate on
the unit sphere $\mathbb{S}^{n-1}$.
The metric $g$ in $\mathbb{H}^{n}$ is
given by
\[
g=dr^{2}+(\sinh r)^{2}d\omega,
\]
with $d\omega^{2}$ representing the canonical metric on the sphere
$\mathbb{S}^{n-1}$.
Following the formula \eqref{distance} we have that $r = d((\cosh r,\omega\sinh r),O)$ is the geodesic distance from a point $x=(\cosh r,\omega\sinh r) \in \mathbb{H}^n$
to origin $O$.
The associated volume form $d\mu$ is
\[
d\mu=dVol_{g}=\sinh^{n-1}rdrd\omega,
\]
where $d\omega:=d\omega_{\mathbb{S}^{n-1}}$ denotes the standard volume form
on $\mathbb{S}^{n-1}$. The Laplace-Beltrami operator $\Delta_{\mathbb{H}^{n}}$
on hyperbolic space can be expressed as
\[
\Delta_{x}:=\Delta_{\mathbb{H}^{n}}=\partial_{r}^{2}+(n-1)\coth r\partial
_{r}+\sinh^{-2}r\Delta_{\mathbb{S}^{n-1}},
\]
where $\Delta_{\mathbb{S}^{n-1}}$ is the Laplace-Beltrami operator on
$\mathbb{S}^{n-1}$. It is well-established that the spectrum of the negative
Laplace-Beltrami operator $-\Delta_{x}$ consists of the half-line $[\rho
^{2},\infty)$, where $\rho=\dfrac{n-1}{2}$.

For $\omega\in \mathbb{S}^{n-1}$ and $\lambda\in \mathbb{R}$, let $b(\omega) = (1,\omega)\in \mathbb{R}^{n+1}$ and $h_{\lambda,\omega}: \mathbb{H}^n\to \mathbb{C}$ is given by
$$h_{\lambda,\omega}(x) = \left< x,b(\omega)\right>_g^{i\lambda-\rho}.$$
The functions $h_{\lambda,\omega}$ satisfies
$$-\Delta_x h_{\lambda,\omega} = (\lambda^2+\rho^2)h_{\lambda,\omega},$$
then they are generalized eigenfunctions of $\Delta_x$. Using these functions one can determine the Fourier transform of a continuous function with compactly support $f\in C_0(\mathbb{H}^n)$ by
\begin{equation}\label{Fou1}
\hat{f}(\lambda,\omega) = \int_{\mathbb{H}^n}f(x)h_{\lambda,\omega}(x)d\mu.
\end{equation}
Moreover, if $f\in C_0^\infty(\mathbb{H}^n)$ then the Fourier inversion formula of $f$ is
\begin{align}\label{Fou2}
f(x) &= \int_{\mathbb{R}} \int_{\mathbb{S}^{n-1}} \hat{f}(\lambda,\omega)\bar{h}_{\lambda,\omega}(x)|{\bf c}(\lambda)|^{-2}d\lambda d\omega \notag\\
&= \int_{\mathbb{R}}\int_{\mathbb{S}^{n-1}} \hat{f}(\lambda,\omega)\left< x,b(\omega)\right>_g^{-i\lambda-\rho}|{\bf c}(\lambda)|^{-2}d\lambda d\omega,
\end{align}
where $c(\lambda)$ is the Harish-Chandra coefficient
\begin{equation}\label{Har}
|{\bf c}(\lambda)|^{-2} = {\bf c}(\lambda)^{-1} {\bf c}(-\lambda)^{-1} = C\frac{|\Gamma(i\lambda +\rho)|^2}{|\Gamma(i\lambda)|^2},
\end{equation}
for a suitable constant $C$. The Harish-Chandra coefficient can be estimated as (see \cite[Proposition A1]{Ionescu2000}):
\begin{equation}\label{HarEs}
|\partial_\lambda^\alpha (\lambda^{-1}{\bf c}(\lambda)^{-1})| \leq C (1+|\lambda|)^{\rho-1-\alpha}
\end{equation}
for $\alpha \in [0,\, n+2]\cap \mathbb{Z}$ and $\lambda\in \mathbb{R}$.

The Fourier transform $f\to \hat{f}$  extends to an isometry from $L^2(\mathbb{H}^n)$ to $L^2(\mathbb{R}_+\times \mathbb{S}^{n-1},\, |{\bf c}(\lambda)|^{-2}d\lambda d\omega)$ since the Plancherel theorem is valid on the hyperbolic space $\mathbb{H}^n$. In addition, for all functions $f_1,f_2 \in L^2(\mathbb{H}^{n})$, we have
$$\int_{\mathbb{H}^n}f_x(x)\overline{f_2(x)}d\mu = \frac{1}{2}\int_{\mathbb{R}\times \mathbb{S}^{n-1}}\hat{f}_1(\lambda,\omega)\overline{\hat{f}_2(\lambda,\omega)}|{\bf c}(\lambda)|^{-2}d\lambda d\omega.$$

Moreover, if we consider that $f$ is $SO(n)$-invariant, then the formulas \eqref{Fou1} and \eqref{Fou2} become
\begin{equation}\label{Fou3}
\hat{f}(\lambda,\omega)=\hat{f}(\lambda)=\int_{\mathbb{H}^n}f(x){\bf \Phi}_{-\lambda}(x)d\mu
\end{equation}
and
\begin{equation}\label{Fou4}
f(x) = \int_{\mathbb{R}} \hat{f}(\lambda){\bf \Phi}_{\lambda}(x)|{\bf c}(\lambda)|^{-2}d\lambda,
\end{equation}
where
\begin{equation}
{\bf \Phi}_{\lambda}(x) = \int_{\mathbb{S}^{n-1}}\left< x,b(\omega)\right>^{-i\lambda-\rho}_gd\omega
\end{equation}
is the spherical function.

By setting $x=(\cosh r, \mathcal{V}\sinh r) \in \mathbb{H}^n$, where $\mathcal{V}\in \mathbb{S}^{n-1}$ and $r=r(x)=d(x,O)$ is the geodesic distance (with respect to metric $g$) from the point $x$ to origin $O$, one can calculate to find that (for more details see \cite[Lemma 3.2]{Ba2007})
\begin{align}\label{Phi}
{\bf \Phi}_{\lambda}(x) &= {\bf \Phi}_\lambda(r(x)) = \int_{\mathbb{S}^{n-1}} \left( \cosh r - (\mathcal{V}\cdot \omega)\sinh r\right)^{-i\lambda-\rho}d\omega \nonumber\\
 &= C\int_0^\pi \left(\cosh r - \sinh r \cos \theta \right)^{-i\lambda-\rho}(\sin\theta)^{n-2}d\theta.
\end{align}
In what follows, we recall some decomposition properties of ${\bf \Phi}_\lambda$ which will be useful to estimate the kernel of the GBq-prototype group in Subsection \ref{S31} below (see \cite[Proposition A2]{Ionescu2000} for the proof):
\begin{itemize}
\item[$\bullet$] For $r\geq \dfrac{1}{10}$, we can write ${\bf \Phi}_\lambda$ as follows
\begin{equation}\label{Sph1}
{\bf \Phi}_\lambda(r) = e^{-\rho r} \left( e^{i\lambda r}{\bf c}(\lambda)m_1(\lambda,r) + e^{-i\lambda r}{\bf c}(-\lambda)m_1(-\lambda,r) \right),
\end{equation}
where the function $m_1(\lambda,r)$ satisfies the following estimate
\begin{equation}\label{Sphere1}
|\partial^\alpha_\lambda m_1(\lambda,r)| \leq C(1+|\lambda|)^{-\alpha}
\end{equation}
with $\alpha \in [0,\, n+2]\cap \mathbb{Z}$ and $\lambda \in \mathbb{R}$.

\item[$\bullet$] For $r \leq 1$, we can write ${\bf \Phi}_\lambda$ as follows
\begin{equation}\label{Sph2}
{\bf \Phi}_\lambda(r) = e^{i\lambda r} m_2(\lambda,r) + e^{-i\lambda r} m_2(-\lambda,r),
\end{equation}
where the function $m_2(\lambda,r)$ satisfies the following estimate
\begin{equation}\label{Sphere2}
|\partial^\alpha_\lambda m_2(\lambda,r)| \leq C(1+r|\lambda|)^{-\rho}(1+|\lambda|)^{-\alpha}
\end{equation}
with $\alpha \in [0,\, n+2]\cap \mathbb{Z}$ and $\lambda \in \mathbb{R}$.
\end{itemize}

\subsection{Integral formulation and associated groups}\label{S22}

The IVP (\ref{Bou}) is equivalent to the following system of equations
\begin{equation}%
\left\{
\begin{array}
[c]{rcl}%
u_t &=& v, \text{  } x\in \mathbb{H}^n, \, t\in \mathbb{R},\\
v_t &=& u - \Delta_x u - f(u), \text{  } x\in \mathbb{H}^n,\, t\in \mathbb{R},\\
 u(0,x)&=&u_{0}(x), \text{  }  x\in \mathbb{H}^n,\\
u_t(0,x)&=& \phi(x) = \Delta_x v_0(x), \text{  } x\in \mathbb{H}^n.
\end{array}
\right.
\label{Bou1}%
\end{equation}
In view of Duhamel's principle, we can formally convert the system (\ref{Bou1}) into the integral equation

\begin{equation}
\left[u(t), v(t) \right]= G(t)\left[u_{0},v_{0}\right] - \int_{0}^{t}G(t-s)\left[0,f(u(s)) \right]ds,
\label{intergralEq}%
\end{equation}
where the Boussinesq group $G(t)$ is determined by the associated linear system of \eqref{Bou1}:
\begin{equation}
\left\{
\begin{array}
[c]{rcl}%
u_{t} & = & \Delta v,\text{ }x\in\mathbb{H}^{n},t\in\mathbb{R},\\
v_{t} & = & u-\Delta u,\ x\in\mathbb{H}^{n},t\in\mathbb{R},\\
u(x,0) & = & u_{0}(x),\text{ }x\in\mathbb{H}^{n},\\
v(x,0) & = & v_{0}(x),\ x\in\mathbb{H}^{n}.
\end{array}
\right.  \label{linear11}%
\end{equation}
More precisely, for an initial data $\left[  u_{0},v_{0}\right]  $ and $t\in\mathbb{R}$, the
linear operator $G(t)$ is given by%
\begin{equation}
G(t)[u_{0},v_{0}]=\int_{\mathbb{R}}\left[
\begin{array}
[c]{cc}%
\cos(t|\xi|\left\langle \xi\right\rangle ) & -|\xi|\left\langle \xi
\right\rangle ^{-1} \sin(t|\xi|\left\langle \xi\right\rangle )\\
\left\langle \xi\right\rangle |\xi|^{-1}\sin(t|\xi|\left\langle \xi
\right\rangle ) & \cos(t|\xi|\left\langle \xi\right\rangle )
\end{array}
\right]  \left[
\begin{array}
[c]{c}%
\!\hat{u_{0}}(\xi)\!\\
\!\hat{v_{0}}(\xi)\!
\end{array}
\right]  \mathbf{\Phi}_{\lambda}(r)|\mathbf{c}(\lambda)|^{-2}d\lambda
,\label{aux-Bou-rel-1}%
\end{equation}
where
\begin{equation}
|\xi|=\sqrt{\lambda^{2}+\rho^{2}} \text{    and  } \left\langle \xi\right\rangle =(1+|\xi|^{2})^{1/2}.
\end{equation}
Equivalently, we can express
\begin{equation}
G(t)\left[  u_{0},\,v_{0}\right]  =\left[  g_{1}(t)u_{0}+g_{2}(t)v_{0}%
,\,g_{3}(t)u_{0}+g_{4}(t)v_{0}\right]  \label{BouOP}%
\end{equation}
where $g_{1}(t)=g_{4}(t)$ and $g_{1}(t),\,g_{2}(t),\,g_{3}(t)$ are the
multiplier operators with the respective symbols
\begin{equation}
\cos(t|\xi|\left\langle \xi\right\rangle ),\,-|\xi|\left\langle \xi
\right\rangle ^{-1}\sin(t|\xi|\left\langle \xi\right\rangle ),\text{ and
\ }\left\langle \xi\right\rangle |\xi|^{-1}\sin(t|\xi|\left\langle
\xi\right\rangle ).\label{aux-symbol-1}%
\end{equation}

In another way, by Gustafson et al. \cite{Gu2006} (see also Kishimoto \cite{Ki2013}), we can rewrite the first equation of system (\ref{Bou}) as
\begin{equation}
\left( i\partial_t - \sqrt{-\Delta_x(1-\Delta_x)} \right) \left(i\partial_t + \sqrt{-\Delta_x(1+\Delta_x)} \right)u = -\Delta_x f(u(t,x)).
\end{equation}
Setting
\begin{equation}
z = u + i (-\Delta_x)^{-\frac{1}{2}}(1-\Delta_x)^{-\frac{1}{2}}\partial_t u = u+ iP^{-1}\partial_t u,
\end{equation}
where $P=\sqrt{-\Delta_x(1-\Delta_x)}$, the IVP (\ref{Bou}) becomes
\begin{equation}\label{Sch}
\begin{cases}
i\partial_t z - Pz - Qf(\Re(z)) = 0,\\
z(0,x) = z_0(x) = u_0 + iP^{-1}\phi,
\end{cases}
\end{equation}
where
\begin{equation}
Q = \sqrt{\frac{-\Delta_x}{1-\Delta_x}},\, u=\frac{1}{2}(z+\bar{z})=\Re(z),\, \text{and }\, P^{-1}\partial_t u = \frac{i}{2}(\bar{z}-z) = \Im(z).
\end{equation}
Thus, by Duhamel's principle, the system \eqref{Sch} can be rewritten as
\begin{equation}
z(t) = e^{-itP}z_0 - i\int_0^t e^{-i(t-s)P}Q(f(\Re(z)))(s)ds,
\end{equation}
where $e^{-itP}$ is the GBq-prototype group associated with the linearized modified GBq-equation
\begin{equation}\label{ModiSchro}
i\partial_t z - Pz =0.
\end{equation}

\section{Dispersive estimates}\label{S3}
\subsection{Pointwise estimates for the GBq-prototype group}\label{S31}
In this subsection we establish the pointwise estimates, i.e., time and radial decays at a fixed point $(x,t)\in \mathbb{H}^n \times \mathbb{R}$ for the GBq-prototype group associating with equation \eqref{ModiSchro}.
Applying the Fourier transform to \eqref{ModiSchro}, we obtain
\begin{equation}\label{FouSchr}
i\partial_t \hat{z} - \sqrt{(\lambda^2+\rho^2)(\lambda^2+\rho^2+1)}\hat{z} = 0.
\end{equation}
The Cauchy problem of equation \eqref{FouSchr} with initial data $\hat{z}_0$ has unique solution is given by
\begin{eqnarray}
\hat{z}(t,\lambda,\omega) = e^{-it \sqrt{(\lambda^2+\rho^2)(\lambda^2+\rho^2+1)}}\hat{z}_0.
\end{eqnarray}

By using the Fourier inversion formula \eqref{Fou3}, the solution of the Cauchy
problem of equation \eqref{ModiSchro} with initial data $z_{0}$ is
\begin{equation}
z(t,r)=e^{-itP}z_{0}=\int_{\mathbb{R}}e^{-it|\xi|\left\langle \xi\right\rangle
}\hat{z}_{0}(\lambda)\mathbf{\Phi}_{\lambda}(r)|\mathbf{c}(\lambda
)|^{-2}d\lambda, \label{SchroOP}%
\end{equation}
where $|\xi|=\sqrt{\lambda^{2}+\rho^{2}}$.

From formulas \eqref{BouOP} and \eqref{SchroOP}, we have the following
relation between $l_{1}(t)=g_{1}(t)$ (which is the multiplier operator with
the symbol $\cos(t|\xi|\left\langle \xi\right\rangle )$) and $e^{-itP}$
\begin{equation}
l_{1}(t)\varphi=\frac{1}{2}\left(  e^{-itP}+e^{itP}\right)  \varphi,
\label{aux-g1}%
\end{equation}
for all scalar functions $\varphi:\mathbb{H}^{n}\rightarrow\mathbb{R}$.
Moreover, if we denote by $l_{2}(t)$ the multiplier operator with the symbol
$\sin(t|\xi|\left\langle \xi\right\rangle )$, then we also have the relation
\begin{equation}
l_{2}(t)\varphi=\frac{1}{2i}\left(  -e^{-itP}+e^{itP}\right)  \varphi,
\label{aux-g2}%
\end{equation}
for all scalar functions $\varphi:\mathbb{H}^{n}\rightarrow\mathbb{R}$.

To establish the dispersive estimates of the GBq-prototype group $e^{-itP}$, we regularize the integral (the associate kernel)
$$ I(t,r) = \int_{\mathbb{R}} e^{-it |\xi| \left< \xi \right>} {\bf \Phi}_{\lambda}(r)|{\bf c}(\lambda)|^{-2}d\lambda$$
by using the Fourier multiplier $\lambda \mapsto e^{-\varepsilon^2 \lambda^2}$, for $\varepsilon>0$ small enough. Precisely, we consider the integral
\begin{equation}
I_{\varepsilon}(t,r) = \int_{\mathbb{R}} e^{-it \sqrt{(\lambda^2+\rho^2)(\lambda^2+\rho^2+1)} - \varepsilon^2\lambda^2}{\bf \Phi}_{\lambda}(r)|{\bf c}(\lambda)|^{-2}d\lambda,
\end{equation}
for $\varepsilon>0$ small enough. Setting
$$\psi(\lambda) = \sqrt{(\lambda^2+\rho^2)(\lambda^2+\rho^2+1)} = \sqrt{T(T+1)},$$
where $T=\lambda^2+\rho^2$, by straightforward calculations, we have
\begin{equation}
\frac{d\psi}{d\lambda} = \frac{d\psi}{d T}\frac{dT}{d\lambda}= \frac{(2(\lambda^2+\rho^2)+1)\lambda}{\sqrt{(\lambda^2+\rho^2)^2+\lambda^2+\rho^2}}
\end{equation}
and
\begin{align}\label{d2}
\frac{d^2\psi}{d\lambda^2} &= \frac{d^2\psi}{dT^2}\left(\frac{dT}{d\lambda}\right)^2 + \frac{d\psi}{dT}\frac{d^2T}{d\lambda^2}\nonumber\\
&= -\frac{\lambda^2}{(T^2+T)\sqrt{T^2+T}} + \frac{2T+1}{\sqrt{T^2+T}}\nonumber\\
&=\frac{(2T+1)(T^2+T) - T+\rho^2}{(T^2+T)\sqrt{T^2+T}}\nonumber\\
&> \frac{2T^3+3T^2 +1/4}{(T^2+T)(T+1/2)},
\end{align}
because $\rho^2\geq \dfrac{1}{4}$ for $n\geq 2$ and $\sqrt{T^2+T}<T+\dfrac{1}{2}$. As consequently of \eqref{d2}, we have $\dfrac{d^2\psi}{d\lambda^2}> 1$ for $T\geq \rho^2\geq \dfrac{1}{4}$.

Therefore, for $m\in C^1_0(\mathbb{R})$ we use the van der Corput lemma (see \cite[Page 334]{St1986}) with noting that $\dfrac{d^2\psi}{d\lambda^2}> 1$ to get
\begin{equation}\label{Osci1}
\left| \int_\mathbb{R} e^{i\left(-t\psi(\lambda) \pm r\lambda\right)} m(\lambda)d\lambda \right| \leq \frac{c}{|t|^{1/2}}\int_\mathbb{R} |\partial_\lambda m(\lambda)| d\lambda.
\end{equation}
Note that, inequality \eqref{Osci1} is valid even in the case that the phase $-t\psi(\lambda)\pm r\lambda$ has critical points.

Moreover, let $\widetilde{m}: \mathbb{R} \to \mathbb{C}$ be a smooth function supported in $[-2,\, -1/2] \cup [1/2,\, 2]$ and assume that the phase $-t\psi(\lambda)\pm r\lambda$ has no critical points in the support of $\widetilde{m}$. By integrating by parts $k\in \mathbb{N}$ times with noting that $\partial_\lambda(-t\psi(\lambda)\pm r\lambda)$ is monotonic, smooth and has no critical points, we obtain that (see the proof in Appendix)
\begin{equation}\label{Osci2}
\left|  \int_\mathbb{R} e^{i\left(-t\psi(\lambda) \pm r \lambda \right)} \widetilde{m}(\lambda)d\lambda \right| \leq \frac{c}{(1+D)^k} \sum_{j=0}^k \int_\mathbb{R} |\partial_\lambda^j \widetilde{m}(\lambda)|d\lambda,
\end{equation}
where
\begin{align}
D &= \min\limits_{|\lambda|\in [1/2,\, 2]} \left|-t\dfrac{d\psi}{d\lambda} \right| + r \nonumber\\
&= \min\limits_{|\lambda|\in [1/2,\, 2]} \left| -t\frac{(2(\lambda^2+\rho^2)+1)\lambda}{\sqrt{(\lambda^2+\rho^2)^2+\lambda^2+\rho^2}} \right| + r.
\end{align}

In what follows, we extend the methods in  \cite[Theorems 3.1 and 3.2]{Anker2014} and \cite[Lemma 3.3]{Ionescu2009} to establish the pointwise decay for $I_{\varepsilon}$ uniformly with respect to $\varepsilon$. This decay plays an important role to obtain the dispersive estimates for the GBq-prototype group.
\begin{lemma}\label{KerEst}
The following pointwise estimates for $I_\varepsilon(t,r)$ hold:
\begin{equation}
|I_{\varepsilon}(t,r)| \leq
C\begin{cases}
|t|^{-n/2}r^{(n+3)/4}e^{-(n-1)r/2} &\hbox{   if   } |t|\leq 1;\\
|t|^{-3/2}r^{(n+3)/4}e^{-(n-1)r/2} &\hbox{   if   } |t|\geq 1,
\end{cases}
\end{equation}
uniformly for all $\varepsilon>0$.
\end{lemma}
\begin{proof}
Let $\eta_0:\mathbb{R} \to [0,\, 1]$ be a smooth even function supported in $[-2,\, -1/2]\cup [1/2,\, 2]$ such that for any $\lambda \in \mathbb{R} \setminus \left\{ 0\right\}$, we have
\begin{equation}\label{Decom}
\sum_{j\in \mathbb{Z}} \eta_0(\lambda/2^j)=1.
\end{equation}
For any $j\in \mathbb{Z}$, we set $\eta_j(\lambda) = \eta_0(\lambda/2^j)$ and $\eta_{\leq j} = \sum\limits_{j'\leq j}\eta_{j'}$. In the following, we establish the estimates for $I_\varepsilon(t,r)$ on two domains depending on the radial variable: $\left\{r\geq 1\right\}$ and $\left\{0<r\leq 1\right\}$.\\

\underline{On the domain $\left\{r\geq 1\right\}$.} By using the equalities \eqref{Har}, \eqref{Sph1} and \eqref{Decom}, we have
\begin{align}\label{twoterms}
|I^+_\varepsilon(t,r)_|{_{\left\{r\geq 1\right\}}}| &= e^{-\rho r}\left| \int_{\mathbb{R}} e^{i\Psi(\lambda)-\varepsilon^2 \lambda^2}\eta_{\leq J}(\lambda) m_1(\lambda,r){\bf c(-\lambda)}^{-1}d\lambda  + \sum_{j>J} \int_{\mathbb{R}} e^{i\Psi(\lambda)-\varepsilon^2 \lambda^2}\eta_{j}(\lambda)m_1(\lambda,r){\bf c(-\lambda)}^{-1}d\lambda  \right| \nonumber\\
&\leq e^{-\rho r}\left| \int_{\mathbb{R}} e^{i\Psi(\lambda)-\varepsilon^2 \lambda^2}\eta_{\leq J}(\lambda) m_1(\lambda,r){\bf c(-\lambda)}^{-1}d\lambda  \right| \nonumber\\
&+ e^{-\rho r}\sum_{j>J}\left| \int_{\mathbb{R}} e^{i\Psi(\lambda)-\varepsilon^2 \lambda^2}\eta_{j}(\lambda)m_1(\lambda,r){\bf c(-\lambda)}^{-1}d\lambda  \right| \nonumber\\
&= e^{-\frac{(n-1)r}{2}}\left(I + \sum_{j>J}II_j\right),
\end{align}
where $\Psi(\lambda) = -t\psi(\lambda)+r\lambda$ and $J$ is the smallest integer such that $2^J\geq 2^{4}\lambda_{r,t}$, with
the function $\lambda_{r,t}$ being selected based on the asymptotic behavior of the roots of
$$\dfrac{d\Psi}{d\lambda} = t\left(-\dfrac{(2(\lambda^2+\rho^2)+1)\lambda}{\sqrt{(\lambda^2+\rho^2)^2+\lambda^2+\rho^2}}+\dfrac{r}{t}\right)=0 \Leftrightarrow \frac{(2(\lambda^2+\rho^2)+1)^2\lambda^2}{(\lambda^2+\rho^2)^2+\lambda^2+\rho^2} = \frac{r^2}{t^2}.$$
This leads to
$$ \frac{1}{2}\lambda^2 < \frac{4(\lambda^2+\rho^2)}{\lambda^2+\rho^2+1}\lambda^2\leq \frac{r^2}{t^2}\leq \frac{4(\lambda^2+\rho^2+1)}{\lambda^2+\rho^2}\lambda^2 <20\lambda^2.$$
This shows that $\lambda^2\simeq \dfrac{r^2}{t^2}$ and we can choose
\begin{equation}\label{Est2^J}
\lambda_{r,t}=\left|\dfrac{r}{t}\right|.
\end{equation}

Using inequalities \eqref{HarEs} and \eqref{Sphere1}, we have that
\begin{align}\label{m1Est}
\left|\partial_\lambda^k \left(m_1(\lambda,r){\bf c}(-\lambda)^{-1} \right) \right| & = \left|\partial_\lambda^k \left(m_1(\lambda,r)\lambda\lambda^{-1}{\bf c}(-\lambda)^{-1} \right) \right| \nonumber\\
&\leq C \sum_{k_1+k_2+k_3=k} \left|\left(\partial_\lambda^{k_1} m_1(\lambda,r)\right) \left( \partial_{\lambda}^{k_2}\lambda \right)\partial_\lambda^{k_3}\left(-\lambda^{-1}{\bf c}(-\lambda)^{-1} \right) \right| \nonumber\\
&\leq C \sum_{k_1+k_2+k_3=k} (1+|\lambda|)^{\rho-1-k_1-k_3}\left| \partial_\lambda^{k_2}\lambda\right|\nonumber\\
&\leq C \left[ |\lambda|(1+|\lambda|)^{\rho-1-k} + \mathbb{I}_{k\geq 1} (1+|\lambda|)^{\rho-k} \right]
\end{align}
for any integer number $k\in [0,\, n+2]$.

According to the definition of $J$, the phase function $\mathbf{\Psi}(\lambda)$ has critical points only within the region where $\eta_{\leq J}$ is supported. These critical points arise solely from the first term $I$ in the equality \eqref{twoterms}. Therefore, we use the oscillatory integral estimate \eqref{Osci1}, inequality \eqref{m1Est} for $k=0,\, 1$, and the fact that the support of $\eta_{j}\, (j\leq J)$ belongs to $[-2^J,\, 2^J]$, in order to estimate
\begin{align}\label{EstI}
I&= \left| \int_{\mathbb{R}} e^{i\Psi(\lambda)-\varepsilon^2 \lambda^2}\eta_{\leq J}(\lambda) m_1(\lambda,r){\bf c(-\lambda)}^{-1}d\lambda  \right| \nonumber\\
&=\left| \int_{\mathbb{R}} e^{i\Psi(\lambda)}\eta_{\leq J}(\lambda) e^{-\varepsilon^2\lambda^2} m_1(\lambda,r){\bf c(-\lambda)}^{-1}d\lambda  \right| \nonumber\\
&\leq C |t|^{-1/2} \int_{\mathbb{R}} \left| \partial_\lambda \left( \eta_{\leq J} (\lambda)e^{-\varepsilon^2\lambda^2}m_1(\lambda,r){\bf c}(-\lambda)^{-1}\right) \right| d\lambda\nonumber\\
&\leq C |t|^{-1/2} \int_{\mathbb{R}} \left| \partial_\lambda \left( \eta_{\leq J}(\lambda) e^{-\varepsilon^2\lambda^2}\right) m_1(\lambda,r){\bf c}(-\lambda)^{-1} + \eta_{\leq J}(\lambda) e^{-\varepsilon^2\lambda^2}\partial_\lambda \left( m_1(\lambda,r){\bf c}(-\lambda)^{-1} \right) \right| d\lambda\nonumber\\
&\leq C |t|^{-1/2} \left( 2^J (1+2^J)^{\rho-1} \int_0^{2^J}\left|\partial_\lambda \left( \eta_{\leq J}(\lambda) e^{-\varepsilon^2\lambda^2}\right) \right|d\lambda  + \int_0^{2^J} \left( \lambda(1+\lambda)^{\rho-2} + (1+\lambda)^{\rho-1} \right) d\lambda  \right) \nonumber\\
&\leq C |t|^{-1/2} \begin{cases}
2^{\rho J} &\hbox{   if   } J\geq 0;\\
2^J &\hbox{   if   } J \leq 0.
\end{cases}
\end{align}
Above, we also have used the boundedness $\eta_{\leq J}(\lambda) e^{-\varepsilon^2\lambda^2}<1 $ and the fact that
\begin{align}\label{absolute}
\int_0^{2^J}\left|\partial_\lambda \left( \eta_{\leq J}(\lambda) e^{-\varepsilon^2\lambda^2}\right) \right|d\lambda = \left.\left(\mathrm{sgn}\left(\partial_\lambda \left( \eta_{\leq J}(\lambda) e^{-\varepsilon^2\lambda^2}\right)\right)\int\partial_\lambda \left( \eta_{\leq J}(\lambda) e^{-\varepsilon^2\lambda^2}\right)d\lambda \right)\right|_0^{2^J}<\infty.
\end{align}
Using \eqref{EstI}, we obtain that for $\dfrac{r}{|t|}\geq 1$ (hence $J\geq 0$), we have
\begin{align}\label{r/t>1}
I &\leq C|t|^{-1/2}2^{\rho J} \leq C2^{11}|t|^{-1/2}\lambda_{r,t}^\rho \notag\\
&\leq \tilde{C}|t|^{-1/2} \left( \frac{r}{|t|} \right)^{\rho/2} \leq \tilde{C}r^{\rho/2} |t|^{-\frac{n+1}{4}}\notag\\
& \leq \tilde{C}r^{\rho/2} |t|^{-\frac{n}{2}}.
\end{align}
The last inequality holds for $|t|\leq 1$. On the other hand, for $\dfrac{r}{|t|}\leq 1$ (hence $|t|\geq 1$), we have
\begin{align}\label{r/t<1}
I &\leq C|t|^{-1/2}2^J \leq C 2^{11} |t|^{-1/2}\lambda_{r,t} \notag\\
&\leq \hat{C} r|t|^{-3/2}.
\end{align}
Here, we have also used inequality \eqref{Est2^J} and the property of $J$ in inequalities \eqref{r/t>1} and \eqref{r/t<1}.


To estimate the terms $II_j \, (j\geq J+1)$, we first notice that by definition of $J$, the phase ${\bf \Psi}(\lambda)$ does not have any critical point on the support of $II_j$. Hence, we can use the oscillatory integral estimate \eqref{Osci2}, inequality \eqref{m1Est} and a change of variables $\lambda\mapsto 2^j\lambda$ to obtain that
\begin{align}
II_j &= \left| \int_{\mathbb{R}} e^{i\Psi(\lambda)-\varepsilon^2 \lambda^2}\eta_{j}(\lambda)m_1(\lambda,r){\bf c(-\lambda)}^{-1}d\lambda  \right|\nonumber\\
&\leq \frac{C}{\left(1+ 2^{j}\min\limits_{|\lambda|\in [2^{j-1},\, 2^{j+1}]} (|-t\partial_\lambda \psi(\lambda)|+r) \right)^n}2^j\int_{1/2}^2\sum_{k=0}^n \left| \partial_\lambda^k \left( \eta_0(\lambda)e^{-\varepsilon^2(2^j\lambda)^2} m_1(2^j\lambda,r){\bf c}(-2^j\lambda)^{-1}  \right)  \right|d\lambda\nonumber\\
&\leq \frac{C}{\left(1+ 2^{j}\left( r+ 2^{j-1}|t|\right)\right)^n} \sum_{k=0}^n 2^j \left( 2^j(1+2^j)^{\rho-1-k}  +(1+2^j)^{\rho-k}\right)\nonumber\\
&\leq \frac{C}{\left(1+ 2^jr+ 2^{2j}\dfrac{|t|}{2}\right)^n} \begin{cases}
2^{j(\rho+1)} &\hbox{   if   } j\geq 0;\\
2^{j} &\hbox{   if   } j \leq 0.
\end{cases}
\end{align}
Above, we have used the facts that $\left|\partial_\lambda^k(\eta_0(\lambda)e^{-\varepsilon^2(2^j\lambda)^2})\right|<\infty$ and $\int_{1/2}^2\left|\partial_\lambda^k(\eta_0(\lambda)e^{-\varepsilon^2(2^j\lambda)^2})\right|d\lambda <\infty$ since $\eta_0(\lambda)$ is smooth and compactly supported.
Therefore, if $J <0$ and $r\geq 1$, then $1\leq r \leq |t|$ and we have that
\begin{align}\label{est-II-1}
\sum_{j=J+1}^0 II_j &\leq \sum_{j=J+1}^0 \frac{2^j}{\left( 1+ 2^{2j}\dfrac{|t|}{2}\right)^n} \leq \sum_{j=J+1}^0 \frac{2^j}{\left( 1+ 2^{2j}\dfrac{|t|}{2}\right)^{1/2}} \nonumber\\
&\leq C\sum_{j=J+1}^0 \frac{2^j}{\left( 2^j|t|^{1/2}\right)^{1/2}} \leq C|t|^{-1/4}\sum_{j=J+1}^0 2^{j/2}\notag\\
& \leq C|t|^{-1/4}.
\end{align}
If $j\geq 0$, then we have
\begin{align}\label{est-II-2}
\sum_{j\geq 0} II_j &\leq \sum_{j\geq 0} \frac{C2^{\frac{j(n+1)}{2}}}{\left( 2^j + 2^{2j}\frac{|t|}{2} \right)^n} \leq C\left( 1 + \frac{|t|}{2} \right)^{-n}\sum_{j\geq 0} 2^{\frac{(1-n)j}{2}}  \notag\\
&\leq C\left( 1 + \frac{|t|}{2} \right)^{-n} \notag\\
&\leq C|t|^{-\frac{n}{2}}.
\end{align}

Combining inequalities \eqref{r/t>1}, \eqref{r/t<1}, \eqref{est-II-1} and \eqref{est-II-2}, we obtain the dispersive estimates on domain $\left\{ r\geq 1\right\}$ as follows
\begin{align}\label{est-r>1}
|I^+_\varepsilon(t,r)_{|_{\left\{ r\geq 1\right\}}}| &\leq C \left( |t|^{-\frac{n}{2}} + |t|^{-\frac{1}{4}}\right)(1+r)^{\rho/2+1} e^{-\frac{(n-1)r}{2}}\notag\\
&= C \left( |t|^{-\frac{n}{2}} + |t|^{-\frac{1}{4}} \right)(1+r)^{\frac{n+3}{4}} e^{-\frac{(n-1)r}{2}}\notag\\
&= \begin{cases}
C|t|^{-\frac{n}{2}} (1+r)^{\frac{n+3}{4}} e^{-\frac{(n-1)r}{2}}&\hbox{   if   }  0<|t| \leq 1;\\
C|t|^{-\frac{1}{4}}(1+r)^{\frac{n+3}{4}} e^{-\frac{(n-1)r}{2}} &\hbox{   if   } |t| \geq 1.
\end{cases}
\end{align}
At the end of proof, we are going to improve the decay to $|t|^{-3/2}$ for $|t|\geq 1$. By the same way, we also can establish that
\begin{equation}\label{est-r>1'}
|I^-_\varepsilon(t,r)_{|_{\left\{ r\geq 1\right\}}}|\leq \begin{cases}
C|t|^{-\frac{n}{2}} (1+r)^{\frac{n+3}{4}} e^{-\frac{(n-1)r}{2}}&\hbox{   if   }  0<|t| \leq 1;\\
C|t|^{-\frac{1}{4}}(1+r)^{\frac{n+3}{4}} e^{-\frac{(n-1)r}{2}} &\hbox{   if   } |t| \geq 1,
\end{cases}
\end{equation}
where $I^-_\varepsilon(t,r)_{|_{\left\{ r\geq 1\right\}}}$ is given by the same formula as $I^+_\varepsilon(t,r)_{|_{\left\{ r\geq 1\right\}}}$ but with $m_1(\lambda,r){\bf c}(-\lambda)^{-1}$ replaced by $m_1(-\lambda,r){\bf c}(\lambda)^{-1}$.

\underline{On the domain $\left\{0<r\leq 1\right\}$.}
By using the equalities \eqref{Sph2} and \eqref{Decom}, we arrive at
\begin{align}\label{TwoTerms}
|I^+_\varepsilon(t,r)_|{_{\left\{0<r\leq 1\right\}}}|&= \left| \int_{\mathbb{R}} e^{i\Psi(\lambda)-\varepsilon^2 \lambda^2}\eta_{\leq J}(\lambda) m_2(\lambda,r)|{\bf c(\lambda)}|^{-2}d\lambda  \right. \nonumber\\
&\left.+ \sum_{j>J} \int_{\mathbb{R}} e^{i\Psi(\lambda)-\varepsilon^2 \lambda^2}\eta_{j}(\lambda)m_2(\lambda,r)|{\bf c(\lambda)}|^{-2}d\lambda  \right| \nonumber\\
&\leq \left| \int_{\mathbb{R}} e^{i\Psi(\lambda)-\varepsilon^2 \lambda^2}\eta_{\leq J}(\lambda) m_2(\lambda,r)|{\bf c(\lambda)}|^{-2}d\lambda  \right| \nonumber\\
&+ \sum_{j>J}\left| \int_{\mathbb{R}} e^{i\Psi(\lambda)-\varepsilon^2 \lambda^2}\eta_{j}(\lambda)m_2(\lambda,r)|{\bf c(\lambda)}|^{-2}d\lambda  \right| \nonumber\\
&= \widetilde{I} + \sum_{j>J} \widetilde{II}_j,
\end{align}
where $\Psi(\lambda) = -t\psi(\lambda)+r\lambda$ and $J$ is again the smallest integer such that ${2^J\geq 2^{4}\lambda_{r,t}}$, with $\lambda_{r,t}$ given by \eqref{Est2^J}.

By using inequalities \eqref{HarEs} and \eqref{Sphere2}, we have
\begin{align}\label{m2Est}
\left|\partial_\lambda^k \left(m_2(\lambda,r)|{\bf c}(\lambda)|^{-2} \right) \right| & = \left|\partial_\lambda^k \left(m_2(\lambda,r) {\bf c}(\lambda)^{-1}{\bf c}(-\lambda)^{-1} \right) \right| \nonumber\\
&\leq C \sum_{k_1+k_2+k_3+k_4=k} \left|\left(\partial_\lambda^{k_1} m_2(\lambda,r)\right) \partial_{\lambda}^{k_2}(\lambda^2) \partial_\lambda^{k_3}\left(\lambda^{-1}{\bf c}(\lambda)^{-1} \right) \partial_\lambda^{k_4}\left( -\lambda^{-1}{\bf c}(-\lambda)^{-1} \right) \right| \nonumber\\
&\leq C \sum_{k_1+k_2+k_3+k_4=k} (1+r|\lambda|)^{-\rho} (1+|\lambda|)^{2(\rho-1)-k_1-k_3-k_4}\left|  \partial_\lambda^{k_2}(\lambda^2)\right| \nonumber\\
&\leq C (1+r|\lambda|)^{-\rho} (1+|\lambda|)^{2(\rho-1)} \left( |\lambda|^2(1+|\lambda|)^{-k} + \mathbb{I}_{k_2\geq 1}|\lambda|(1+|\lambda|)^{1-k}\right. \nonumber\\
&\hspace{8cm}\left.+ \mathbb{I}_{k_2\geq 2}(1+|\lambda|)^{2-k} \right)
\end{align}
for any integer number $k\in [0,\, n+2]$.

Now, using inequality \eqref{m2Est} for $k=0,\, 1$, oscillatory integral's estimate \eqref{Osci1} and the fact that the support of $\eta_{j\leq J}$ belongs to $[-2^J,\, 2^J]$, we obtain that
\begin{align}\label{EstI'}
\widetilde{I}&= \left| \int_{\mathbb{R}} e^{i\Psi(\lambda)-\varepsilon^2 \lambda^2}\eta_{\leq J}(\lambda) m_2(\lambda,r)|{\bf c(\lambda)}|^{-2}d\lambda  \right| \nonumber\\
&\leq C |t|^{-1/2} \int_{\mathbb{R}} \left| \partial_\lambda \left( \eta_{\leq J}(\lambda) e^{-\varepsilon^2\lambda^2} m_2(\lambda,r)|{\bf c}(\lambda)|^{-2}\right) \right| d\lambda\nonumber\\
&\leq C |t|^{-1/2} \int_{\mathbb{R}} \left| \partial_\lambda \left( \eta_{\leq J}(\lambda) e^{-\varepsilon^2\lambda^2} \right) m_2(\lambda,r)|{\bf c}(\lambda)|^{-2} + \eta_{\leq J}(\lambda) e^{-\varepsilon^2\lambda^2} \partial_\lambda \left( m_2(\lambda,r)|{\bf c}(-\lambda)|^{-2} \right) \right| d\lambda\nonumber\\
&\leq C |t|^{-1/2} \left( 2^{2J} (1+r2^J)^{-\rho}(1+2^J)^{2(\rho-1)} \int_0^{2^J} \left|  \partial_\lambda\left( \eta_{\leq J} e^{-\varepsilon^2\lambda^2}\right)\right|d\lambda\right.\notag\\
&\hspace{4cm}\left. + \int_0^{2^J} (1+r\lambda)^{-\rho} (1+\lambda)^{2(\rho-1)} \left( \lambda^2(1+\lambda)^{-1} + \lambda \right) d\lambda  \right) \nonumber\\
&\leq C |t|^{-1/2} \begin{cases}
2^{\rho J}r^{-\rho} &\hbox{   if   } J\geq 0,\, r2^J \geq 1\\
2^{2\rho J} &\hbox{   if   } J \geq 0,\, r2^J \leq 1\\
2^{2J} &\hbox{   if   } J \leq 10.
\end{cases}
\end{align}
As in \eqref{absolute}, again we have used the fact that $\eta_{\leq J}e^{-\varepsilon^2\lambda^2}<1$ and $\int_0^{2^J} \left|  \partial_\lambda\left( \eta_{\leq J} e^{-\varepsilon^2\lambda^2}\right)\right|d\lambda<\infty$. Now, we consider three following cases:
\begin{itemize}
\item[$\bullet$] If $\dfrac{r}{|t|} \geq 1$ (hence $|t| \leq 1$) and $2^J r \geq 1$, then $r^3\geq C|t|$ and using \eqref{Est2^J}, we have
\begin{align}\label{case1}
\widetilde{I} &\leq C |t|^{-1/2} 2^{\rho J}r^{-\rho} \leq C |t|^{-1/2}\left( \frac{r}{|t|} \right)^{\rho/2}r^{-\rho} \notag\\
&\leq  C |t|^{-1/2}|t|^{-\rho/2}r^{-\rho/2}\leq  C |t|^{-1/2}|t|^{-\rho/2}|t|^{-\rho/6} = C|t|^{-\frac{2n+1}{6}} \notag\\
&\leq C |t|^{-\frac{n}{2}}.
\end{align}
\item[$\bullet$] If $\dfrac{r}{|t|}\geq 1$  and $2^Jr \leq 1$, then $r^3\leq C|t|$ and using again \eqref{Est2^J}, we have
\begin{align}\label{case2}
\widetilde{I} &\leq C |t|^{-1/2} 2^{2\rho J} \leq C |t|^{-1/2}\left( \frac{r}{|t|} \right)^{\rho} \leq  C |t|^{-1/2}|t|^{-\rho}r^{\rho}\notag\\
&\leq  C |t|^{-1/2}|t|^{-\rho}|t|^{\rho/3} = C|t|^{-\frac{2n+1}{6}} \notag\\
&\leq C |t|^{-\frac{n}{2}}.
\end{align}
\item[$\bullet$] If $\dfrac{r}{|t|}\leq 1$, then we have
\begin{align}\label{case3}
\widetilde{I} &\leq C |t|^{-1/2} 2^{2J} \leq C |t|^{-1/2}\left( \frac{r}{|t|} \right)^{2}\notag\\
& \leq  C \left(|t|^{-3/2} + |t|^{-\frac{n}{2}} \right),
\end{align}
because $|t|^{-5/2}\leq |t|^{-3/2}$ if $t\geq 1$ and $|t|^{-1/2} \leq |t|^{-\frac{2n+1}{6}} \leq |t|^{-\frac{n}{2}}$ if $|t|\leq 1$.
\end{itemize}
Combining inequalities \eqref{EstI'}, \eqref{case1},\eqref{case2} and \eqref{case3}, we obtain that
\begin{equation}\label{TildeI}
\widetilde{I} \leq C\left( |t|^{-3/2} + |t|^{-\frac{n}{2}} \right).
\end{equation}

We remain to estimate the terms $\widetilde{II}_j$ for $j \geq J+1$. By using the oscillatory integral's estimate \eqref{Osci2}, inequality \eqref{m2Est} and a change of variables $\lambda\mapsto 2^j\lambda$, we obtain the
\begin{align}\label{fin1}
\widetilde{II}_j &= \left| \int_{\mathbb{R}} e^{i\Psi(\lambda)-\varepsilon^2 \lambda^2}\eta_{j}(\lambda)m_2(\lambda,r)|{\bf c(\lambda)}|^{-2}d\lambda  \right|\nonumber\\
&\leq \frac{C}{\left(1+ 2^{j}\min\limits_{|\lambda|\in [2^{j-1},\, 2^{j+1}]} (|-t\partial_\lambda \psi(\lambda)|+r)\right)^n}2^j\int_{1/2}^2\sum_{k=0}^n \left| \partial_\lambda^k \left( \eta_0(\lambda) e^{-\varepsilon^2(2^j\lambda)^2} m_2(2^j\lambda,r)|{\bf c}(2^j\lambda)|^{-2}  \right)  \right|d\lambda\nonumber\\
&\leq \frac{C}{\left(1+ 2^{j}\left( r+ 2^{j-1}|t|\right)\right)^n} \sum_{k=0}^n 2^j (1+r2^j)^{-\rho}(1+2^j)^{2(\rho-1)} \left( 2^{2j}(1+2^j)^{-k}  + 2^j(1+2^j)^{1-k}  + (1+2^j)^{2-k} \right)\nonumber\\
&\leq \frac{C}{\left(1+ 2^{j}r + 2^{2j}\dfrac{|t|}{2}\right)^n} \begin{cases}
2^{j(\rho+1)}r^{-\rho} &\hbox{   if   } j\geq 0,\, r2^j\geq 1,\\
2^{j(2\rho+1)} &\hbox{   if   } j \geq 0,\, r2^j \leq 1,\\
2^{3j}  &\hbox{   if   } j \leq 0,
\end{cases}
\end{align}
where again we have used the facts that $\left|\partial_\lambda^k(\eta_0(\lambda)e^{-\varepsilon^2(2^j\lambda)^2})\right|<\infty$
and $\int_{1/2}^2\left|\partial_\lambda^k(\eta_0(\lambda)e^{-\varepsilon^2(2^j\lambda)^2})\right|d\lambda<\infty$.
Since $\eta_0(\lambda)$ is compactly supported in $[-2,\, -1/2] \cup [1/2,\, 2]$, we have that $\widetilde{II}_j$ has compact support in $[-2^{j+1},\, -2^{j-1}] \cup [2^{j-1},\, 2^{j+1}]$. Also, in the case $j\geq 0$, note that $\lim\limits_{j\to +\infty}\widetilde{II}_j =0$.

For the case $J < 0$, we have
\begin{align}
\widetilde{II}_j &\leq C 2^j \int_{1/2}^2 |m_2(2^j\lambda,r)||{\bf c}(2^j\lambda)|^{-2}d\lambda\leq C 2^{3j}.
\end{align}
Therefore, let $J_1$ be the largest integer such that $2^{J_1}\leq |t|^{-1/2}$, we have
\begin{equation}\label{fin2}
\sum_{j=J+1}^{J_1} \widetilde{II}_j \leq C|t|^{-3/2} \leq C|t|^{-n/2},
\end{equation}
where the last inequality holds if $0\leq |t|\leq 1$.
Now, by using inequality \eqref{fin1} and the fact that $2^{2j}|t|\geq 1$ (hence $|t|\geq 1$) for $j\geq J_1+1$, we have
\begin{align}\label{fin3}
\sum_{j=J_1+1}^0\widetilde{II}_j &\leq \sum_{j=J_1+1}^0 \dfrac{C2^{3j}}{\left( 1+ 2^jr+ 2^{2j}\dfrac{|t|}{2} \right)^n} \leq \sum_{j=J_1+1}^0 \dfrac{C2^{3j}}{\left( 2^{2j}\dfrac{|t|}{2} \right)^{3/4}}\notag\\
& \leq C|t|^{-3/4}\sum_{j=J_1+1}^0 2^{3j/2}\notag\\
& \leq C|t|^{-3/4}.
\end{align}
Combining \eqref{fin2} and \eqref{fin3}, we establish for $J < 0$ that
\begin{equation}\label{fin4}
\sum_{J+1}^0 \widetilde{II}_j \leq C\left(|t|^{-3/4}+ |t|^{-n/2}\right).
\end{equation}

For the case $j\geq 1$, we use \eqref{fin1} to get
\begin{align}\label{fin5}
\sum_{j\geq 1}\widetilde{II}_j &\leq \sum_{j\geq 1}\frac{C}{\left( 1+ 2^jr+ 2^{2j}\dfrac{|t|}{2}\right)^n} \begin{cases}
2^{j(\rho+1)}r^{-\rho} &\hbox{   if   } j\geq 1,\, r2^j\geq 1\\
2^{j(2\rho+1)} &\hbox{   if   } j \geq 1,\, r2^j \leq 1
\end{cases}\notag\\
&\leq  \sum_{j\geq 1}\dfrac{C2^{jn}}{\left(1+ 2^jr+ 2^{2j}\dfrac{|t|}{2}\right)^n} \notag\\
&\leq  \sum_{j\geq 1}\dfrac{C2^{jn}}{\left( 1+2^{2j}\dfrac{|t|}{2}\right)^n}.
\end{align}
If $|t|\geq 1$, then combine with inequality \eqref{fin5}, we have
\begin{align}\label{fin6}
\sum_{j\geq 1}\widetilde{II}_j &\leq \sum_{j\geq 1}\dfrac{C2^{jn}}{2^{2jn}|t|^{3/2}} \leq C|t|^{-3/2}\sum_{j\geq 1}2^{-jn} \notag\\
&\leq C|t|^{-3/2}.
\end{align}
If $0<|t|\leq 1$, we consider three following cases:
\begin{itemize}
\item[$\bullet$] If $\dfrac{r}{|t|}\geq \dfrac{\sqrt{\rho^2+1}}{r}\geq \sqrt{\rho^2+1}$, then by using \eqref{Est2^J} we get $2^{J}\geq 2^{10}\dfrac{r}{|t|}\geq 2^{10}\dfrac{\rho^2+1}{r}$ (hence $J>0$) and  we have
\begin{align}\label{fin7}
\sum_{j\geq J+1}\widetilde{II}_j &\leq  \sum_{j\geq J+1}\dfrac{C2^{j(\rho+1)}r^{-\rho}}{\left(1 + 2^{2j}\dfrac{|t|}{2}\right)^n} \leq C\sum_{j\geq J+1} r^{-\rho}|t|^{-n}2^{-j(2n-\rho-1)}\nonumber\\
&\leq  C|t|^{-n/2}\sum_{j\geq J+1} r^{-\rho-\frac{n}{2}}\left( \frac{r}{|t|}\right)^{n/2}2^{-j(2n-\rho-1)}\nonumber\\
&\leq C|t|^{-n/2}\sum_{j\geq J+1} r^{-n+\frac{1}{2}}\left( \frac{r}{|t|}\right)^{n/2}2^{-j(2n-\rho-1)}\nonumber\\
&\leq C|t|^{-n/2} \sum_{j\geq J+1} 2^{\left(n-\frac{1}{2}\right)J} 2^{\frac{Jn}{2}}2^{-j(2n-\rho-1)} \nonumber\\
&\leq C|t|^{-n/2}\sum_{j\geq J+1} 2^{\frac{J(3n-1)}{2}}2^{-\frac{j(3n-1)}{2}} \leq  C|t|^{-n/2}\sum_{j\geq J+1} 2^{\frac{-(j-J)(3n-1)}{2}} \nonumber\\
&\leq C|t|^{-n/2}.
\end{align}
\item[$\bullet$] If $\sqrt{\rho^2+1} \leq \dfrac{r}{|t|}\leq \dfrac{\sqrt{\rho^2+1}}{r}$, then $\dfrac{r^2}{\rho^2+1} \leq |t| \leq \dfrac{r}{\rho^2+1}$  and by using  again \eqref{Est2^J} we get
$$\frac{\rho^2+1}{r} \geq 2^J \geq 2^{10}\dfrac{r}{|t|} \geq 2^{J-1}.$$
Using \eqref{fin5} we obtain that
\begin{align}\label{fin8}
\sum_{j \geq J+1}\widetilde{II}_j &\leq  \sum_{j\geq J+1} \dfrac{C2^{j(2\rho+1)}}{\left( 1+ 2^{2j}\dfrac{|t|}{2} \right)^n} \leq  \sum_{j\geq J+1} \dfrac{C2^{jn}}{\left( 1+ 2^{2j}\dfrac{|t|}{2} \right)^n} \nonumber\\
&\leq C|t|^{-n/2} \sum_{j\geq J+1} C2^{-jn}|t|^{-n/2} \notag\\
&\leq   C|t|^{-n/2} \sum_{j\geq J+1} C2^{-jn}|r|^{-n}.
\end{align}
Observe that, if $t\to 0$ with the decay $r^{2}$, then $r^{-1}\simeq\dfrac{r}{|t|} \leq 2^J$ and inequality \eqref{fin8} leads to
\begin{align}\label{fin9}
\sum\limits_{j\geq J+1}\widetilde{II}_j &\leq C|t|^{-n/2}\sum_{j\geq J+1} 2^{-(j-J)n} \notag\\
&\leq C|t|^{-n/2}.
\end{align}
If $t\to 0$ with the decay $r$, i.e., $|t|\simeq r\to 0$, then $J$ is bounded by $0\leq J \leq 10$. We treat this case ($J$ boundedness) as the one below.

\item[$\bullet$] If $\dfrac{r}{|t|} \leq \sqrt{\rho^2+1}$, then  we have $\dfrac{r}{\sqrt{\rho^2+1}} \leq |t| \leq 1$ and \eqref{Est2^J} leads to $0<J \leq 11 + [\log_2\sqrt{\rho^2+1}]$ and $2^J\geq 2^{10}\sqrt{\dfrac{r}{|t|}}$. We are considering the situation $|t|\simeq r \to 0$. Then, for a given $|t|\ll 1$, there exists a smallest integer $J_2 \gg J+1$ such that $|t|^{-1}\leq 2^{2J_2}$. We interest only the case that $J$ (hence $J_2$) is finite. Therefore,
\begin{align}\label{fin10}
\sum_{j \geq J+1}\widetilde{II}_j &\leq C|t|^{-n/2} \sum_{j\geq J+1} 2^{-jn}|t|^{-n/2} \nonumber\\
&\leq C|t|^{-n/2} \left( \sum_{j = J+1}^{J_2} 2^{-jn}|t|^{-n/2} + \sum_{j \geq J_2+1} 2^{-jn}|t|^{-n/2} \right)\nonumber\\
&\leq C|t|^{-n/2} \left(  \sum_{j = J+1}^{J_2} 2^{-(j-J_2)n} + \sum_{j \geq J_2+1} 2^{-(j-J_2)n}  \right) \nonumber\\
&\leq  C|t|^{-n/2} \left( 2^{(J_2-J)n}-1 + 1\right)\nonumber\\
&\leq C2^{(J_2-J)n}|t|^{-n/2} \leq C|t|^{-n/2}.
\end{align}
\end{itemize}

Combining inequalities \eqref{TildeI}, \eqref{fin4} and \eqref{fin6}-\eqref{fin10} yields
\begin{align}\label{est-r<1}
|I^+_\varepsilon(t,r)_{|_{\left\{ 0<r\leq 1\right\}}}| &\leq C\left(  |t|^{-\frac{3}{4}} + |t|^{-\frac{n}{2}} \right)\notag\\
&= \begin{cases}
C|t|^{-\frac{n}{2}},&\hbox{   if   }  0<|t| \leq 1,\\
C|t|^{-\frac{3}{4}}, &\hbox{   if   } |t| \geq 1.
\end{cases}
\end{align}
By the same way, we can obtain that
\begin{equation}\label{est-r<1'}
|I^-_\varepsilon(t,r)_{|_{\left\{ 0<r\leq 1\right\}}}| \leq \begin{cases}
C|t|^{-\frac{n}{2}},&\hbox{   if   }  0<|t| \leq 1,\\
C|t|^{-\frac{3}{4}}, &\hbox{   if   } |t| \geq 1,
\end{cases}
\end{equation}
where $I^-_\varepsilon(t,r)_{|_{\left\{ 0<r\leq 1\right\}}}$ is given by the same formula as $I^+_\varepsilon(t,r)_{|_{\left\{ 0<r\leq 1\right\}}}$ with $m_2(\lambda,r)|{\bf c}(\lambda)|^{-2}$ replaced by $m_2(-\lambda,r)|{\bf c}(-\lambda)|^{-2}$.

Combining estimates \eqref{est-r>1}-\eqref{est-r>1'} and \eqref{est-r<1}-\eqref{est-r<1'} for $|I_\varepsilon(t,r)|$ on the domains $\left\{r\geq 1\right\}$ and  $\left\{0<r \leq 1 \right\}$, respectively, we get the pointwise estimate for $I_\varepsilon(t,r)$, for all $(t,r)$, as follows:
\begin{align}\label{Disper}
|I_{\varepsilon}(t,r)| &\leq |I^+_{\varepsilon}(t,r)| + |I^-_{\varepsilon}(t,r)|\notag\\
&\leq C\begin{cases}
|t|^{-n/2}r^{(n+3)/4}e^{-(n-1)r/2}, &\hbox{   if   } |t|\leq 1,\\
|t|^{-1/4}r^{(n+3)/4}e^{-(n-1)r/2}, &\hbox{   if   } |t|\geq 1.
\end{cases}
\end{align}

Finally, we improve the factor decay in \eqref{Disper} when the time is very large $|t|\gg 1$ by $|t|^{-3/2}$. Rather than employing a dyadic decomposition, we express $I_{\varepsilon}(t,r)$ as
\begin{align}
I_{\varepsilon}(t,r) &= \int_\mathbb{R}\chi^0_t(\lambda)e^{-it\psi(\lambda)-\varepsilon^2\lambda^2}{\bf \Phi}_\lambda (r)|{\bf c}(\lambda)|^{-2}d\lambda\notag\\
&+ \int_\mathbb{R}\chi^\infty_t(\lambda)e^{-it\psi(\lambda)-\varepsilon^2\lambda^2}{\bf \Phi}_\lambda (r)|{\bf c}(\lambda)|^{-2}d\lambda\notag\\
&= I_0(t,r) + I_\infty(t,r),
\end{align}
where $\chi_t^0(\lambda) = \chi(\sqrt{t}\lambda)$  is an even cut-off function such that
\begin{equation}
\chi(\lambda) =
\begin{cases}
1, &\hbox{   if   } |\lambda|\leq 1,\\
0, &\hbox{   if   } |\lambda|\geq 2,
\end{cases}
\end{equation}
and $\chi_t^\infty(\lambda) = 1-\chi_t^0(\lambda)$.

Using the fact that $|{\bf \Phi}_\lambda(r)| \leq |{\bf \Phi}_0(r)| \leq Ce^{-\rho r}$ and $|{\bf c}(\lambda)|^{-2} \leq C\lambda^2$, we have
\begin{equation}\label{far1}
|I_0(t,r)| \leq C|t|^{-3/2}e^{-\rho r}.
\end{equation}
Moreover, taking integral by parts $L=k_0+k_1+k_2+k_3$ times, we obtain
\begin{align}\label{far2}
e^{\rho r/2}|I_\infty(t,r)| &\leq C|t|^{-L}\int_\mathbb{R}\left| e^{-it\psi(\lambda)}\left( \frac{\partial}{\partial\lambda}\circ\frac{1}{\lambda}\right)^L \left( \chi_t^\infty(\lambda)|{\bf c}(\lambda)|^{-2} e^{ir}\right) \right|d\lambda\nonumber\\
&\leq C|t|^{-L}\int_\mathbb{R} \sum_{k_0+k_1+k_2+k_3=L}|t|^{k_0/2}|\lambda|^{-L-k_1}r^{k_3}\begin{cases}
|\lambda|^2 &\hbox{   if   } |\lambda|\leq 1;\\
|\lambda|^{n-1} &\hbox{   if   } |\lambda|\geq 1
\end{cases}d\lambda\nonumber\\
&\leq \sum_{k_0+k_1+k_2+k_3=L,\, k_0\geq 1}|t|^{-3/2}r^{k_3}\nonumber\\
&+ \sum_{k_1+k_2+k_3=L}|t|^{-L}r^{k_3}\left( \int_{|t|^{-1/2}\leq \lambda \leq 1}|\lambda|^{2-L-k_1}d\lambda + \int_{|\lambda|\geq 1}|\lambda|^{n-1-L-k_3}d\lambda  \right)\nonumber\\
&\leq |t|^{-3/2}(1+r)^L,
\end{align}
provided that $L > \max\left\{n, 3/2\right\}$. Here, we used that $|\lambda|\simeq |t|^{-1/2}$ if $k_0\geq 1$ and $\partial_\lambda\psi(\lambda)\geq C|\lambda|$. The inequalities \eqref{far1} and \eqref{far2} complete our proof.
\end{proof}

\bigskip

In order to deduce $L^{q}(\mathbb{H}^{n})$-estimates for the GBq-prototype group $e^{-itP}$, we need to use an expression of the $L^{q}(\mathbb{H}^{n})$-norm for radial functions in $\mathbb{H}^{n}$ (see the formula in the proof of \cite[Corollary 3.3]{Anker2009} with noting that $L^{q}(\mathbb{H}^{n}) = L^{q,q}(\mathbb{H}^{n})$):
\begin{equation}
\left\Vert f\right\Vert _{L^{q}(\mathbb{H}^{n})}\simeq\left(  \int_{0}%
^{1}(f(r))^{q}r^{n-1}dr\right)^{1/q}+\left(  \int_{1}^{\infty}%
(f(r))^{q}e^{(n-1)r}dr\right)^{1/q}.\label{normLq-1}%
\end{equation}
The norm of $I(t,r)=\lim\limits_{\varepsilon\rightarrow
0}I_{\varepsilon}(t,r)$ in $L^{q}(\mathbb{H}^{n})$ can be estimated as follows.

\begin{lemma}
\label{KerEst1} Let $2<q\leq\infty$. Then, we have that
\begin{equation}
\left\Vert I(t,r)\right\Vert _{L^{q}(\mathbb{H}^{n})}\leq C%
\begin{cases}
|t|^{-n/2}\hbox{   if   }|t|\leq1,\\
|t|^{-3/2}\hbox{   if   }|t|\geq1.
\end{cases}
\label{est-St-Lq}%
\end{equation}

\end{lemma}

\begin{proof} By using the $\varepsilon$-uniform estimates for
$I_{\varepsilon}$ in Lemma \ref{KerEst}, we obtain that
\begin{equation}
|I(t,r)|\leq Cr^{\frac{n+3}{4}}e^{-\frac{(n-1)r}{2}}%
\begin{cases}
|t|^{-\frac{n}{2}} & \hbox{   if   }|t|\leq1,\\
|t|^{-\frac{3}{2}} & \hbox{   if   }|t|>1.
\end{cases}
\label{est-aux-St-1}%
\end{equation}
By using \eqref{normLq-1}, for $f(r)=r^{\frac{n+3}{4}}e^{-\frac{(n-1)r}{2}}$
and  $2<q\leq\infty$, we can see that $\left\Vert f\right\Vert _{L^{q}%
(\mathbb{H}^{n})}<\infty$. This yields the result in our lemma.
\end{proof}

\subsection{Lorentz spaces and dispersive estimates}\label{S32}

We start by providing a brief primer on Lorentz. For further details, we refer
the reader to \cite{BS},\cite{BeLo}. For $\Omega\subset$ $\mathbb{H}^{n}$ and
$1\leq p\leq\infty,$ \ we denote by $L^{p}(\Omega)$ the space of all $L^{p}%
$-integrable functions on $\Omega.$ Let $1<p_{1}<p_{2}\leq\infty,$ $1\leq
d\leq\infty$ and $\theta\in(0,1)$ be such that $\dfrac{1}{p}=\dfrac{1-\theta
}{p_{1}}+\dfrac{\theta}{p_{2}}.$ Considering the interpolation functor
$(\cdot,\cdot)_{\theta,d}$, the Lorentz space $L^{(p,d)}(\Omega)$ can be
recovered as the interpolation space $\left(  L^{p_{1}}(\Omega),L^{p_{2}%
}(\Omega)\right)  _{\theta,d}=L^{(p,d)}(\Omega)$ with the interpolation norm
$\left\Vert \cdot\right\Vert _{(p,d)}$ induced by $(\cdot,\cdot)_{\theta,d}.$
Special cases that are in the family of Lorentz spaces are the Lebesgue one
$L^{p}(\Omega)=L^{(p,p)}(\Omega)$ and the so-called Marcinkiewicz space
$L^{(p,\infty)}(\Omega)$. This last space is also known as weak-$L^{p}%
(\Omega).$

A direct consequence of the reiteration theorem (see \cite[Theorem
3.5.3]{BeLo}) is the interpolation property
\begin{equation}
\left(  L^{(p_{1},d_{1})}(\Omega),L^{(p_{2},d_{2})}(\Omega)\right)
_{\theta,r}=L^{(p,d)}(\Omega),\label{interp1}%
\end{equation}
where $1<p_{1}<p_{2}\leq\infty$, $1\leq d_{1},d_{2},d\leq\infty,$
$0<\theta<1,$ and $\dfrac{1}{p}=\dfrac{1-\theta}{p_{1}}+\dfrac{\theta}{p_{2}%
}.$ Next we recall that the product operador works well in the Lorentz
setting. More precisely, letting $1<p_{1},p_{2},p_{3}\leq\infty$ and $1\leq
d_{1},d_{2},d_{3}\leq\infty$ with $\dfrac{1}{p_{3}}=\dfrac{1}{p_{1}}+\dfrac
{1}{p_{2}}$ and $\dfrac{1}{d_{1}}+\dfrac{1}{d_{2}}\geq\dfrac{1}{d_{3}}$, we
have the H\"{o}lder-type inequality%
\begin{equation}
\left\Vert fg\right\Vert _{{(p_{3},d_{3})}}\leq C\left\Vert f\right\Vert
_{{(p_{1},d_{1})}}\left\Vert g\right\Vert _{{(p_{2},d_{2})}},\label{Holder}%
\end{equation}
where $C>0$ is a universal constant. Finally, for $1\leq d_{1}\leq p\leq
d_{2}\leq\infty,$ we have the following chain of continuous inclusions
\begin{equation}
L^{(p,1)}(\Omega)\subset L^{(p,d_{1})}(\Omega)\subset L^{p}(\Omega)\subset
L^{(p,d_{2})}(\Omega)\subset L^{(p,\infty)}(\Omega).\label{Inclusion}%
\end{equation}


If $f$ is a positive, radially decreasing function on $\mathbb{H}^{n}$, then
its norm in the Lorentz space $L^{p,d}(\mathbb{H}^{n})$ can be computed
explicitly as follows (see \cite[the proof of Corollary 3.3]{Anker2009}):
\begin{equation}
\left\Vert f\right\Vert _{L^{p,d}(\mathbb{H}^{n})}\simeq\left(  \int_{0}%
^{1}(f(r))^{d}r^{\frac{dn}{p}-1}dr\right)  ^{\frac{1}{d}}+\left(  \int
_{1}^{\infty}(f(r))^{d}e^{\frac{d(n-1)}{p}r}dr\right)  ^{\frac{1}{d}%
}\label{normLpr}%
\end{equation}
for $1\leq d<\infty$, and for $d=\infty$
\begin{equation}
\left\Vert f\right\Vert _{L^{p,\infty}(\mathbb{H}^{n})}\simeq\sup
_{0<r<1}r^{\frac{n}{p}}f(r)+\sup_{r\geq1}e^{\frac{n-1}{p}r}f(r).\label{normLp}%
\end{equation}

Using Lemma \ref{KerEst1} and Kunze-Stein inequality we obtain $L^{q^{\prime}}-L^{q}$-dispersive
estimates for the GBq-prototype group $e^{-itP}$ in the following theorem.

\begin{theorem}
\label{Dispersive} Let  $2<q\leq\infty$ and let $q^{\prime}$ stands for its conjugate exponent, i.e., $\frac{1}{q}+\frac{1}{q^{\prime}}=1$. Then, there exists a constant $C>0$ such
that
\begin{equation}
\left\Vert e^{-itP}\right\Vert _{L^{q^{\prime}}(\mathbb{H}^{n})\rightarrow
L^{q}(\mathbb{H}^{n})}\leq C%
\begin{cases}
|t|^{-\frac{n}{2}\left(  1-\frac{2}{q}\right)  } & \hbox{   if   }|t|\leq1,\\
|t|^{-\frac{3}{2}} & \hbox{   if   }|t|>1.
\end{cases}
\label{est-group-Lq-1}%
\end{equation}

\end{theorem}

\begin{proof} By computing the $L^{\infty}(\mathbb{H}^{n})$-norm directly in the integral expression of $e^{-itP}$, we obtain for small $|t|$ that
\[
\left\Vert e^{-itP}\right\Vert _{L^{1}(\mathbb{H}^{n})\rightarrow L^{\infty
}(\mathbb{H}^{n})}\leq\left\Vert I(t,r)\right\Vert _{L^{\infty}(\mathbb{H}%
^{n})}\leq C|t|^{-n/2}.\]
Also, we have that
\[
\left\Vert e^{-itP}\right\Vert _{L^{2}(\mathbb{H}^{n})\rightarrow
L^{2}(\mathbb{H}^{n})}\leq1.
\]
Interpolating between the two previous inequalities yields the result for $|t|\leq1$.

For large $|t|$, we need the following Kunze-Stein inequality (see \cite{Cow1997,Cow1998}):
\begin{align*}
\left\Vert f\ast g\right\Vert _{L^{q}(\mathbb{H}^{n})}&\leq C\left\Vert f\right\Vert_{L^{q,1}(\mathbb{H}^n)}\left\Vert g\right\Vert _{L^{q^{\prime}}(\mathbb{H}^{n})}\\
&\leq  C\left(  \int
_{0}^{1}f(r)r^{\frac{n}{q}-1}dr+\int_{1}^{\infty}f(r)e^{\frac{(n-1)}{q}%
r}dr\right)  \left\Vert g\right\Vert _{L^{q^{\prime}}(\mathbb{H}^{n})},
\end{align*}
for $q>2$, where above we have used the equivalence (\ref{normLpr}). Taking $f=\left\vert I(t,r)\right\vert $ and using (\ref{est-aux-St-1}), this inequality leads us to
\begin{align*}
\left\Vert e^{-itP}\right\Vert _{L^{q^{\prime}}(\mathbb{H}^{n})\rightarrow
L^{q}(\mathbb{H}^{n})}  & \leq C\left(  \int_{0}^{1}\left\vert I(t,r) \right\vert r^{\frac{n}{q}-1}dr+\int_{1}^{\infty}\left\vert I(t,r)\right\vert e^{\frac{(n-1)}{q}r}dr\right)  \\
& \leq C|t|^{-3/2}\left(  \int_{0}^{1}r^{\frac{n}{q}+\frac{n-1}{4}}%
e^{-\frac{(n-1)}{2}r}dr+\int_{1}^{\infty}r^{\frac{n+3}{4}}e^{-(n-1)(\frac
{1}{2}-\frac{1}{q})r}dr\right)  \\
& \leq C|t|^{-3/2},\,
\end{align*}
for $q>2,$ as desired.
\end{proof}

\begin{remark}\label{Strichatz}
As previous works (see for example \cite{Anker2009,Ionescu2009}) we can also obtain a more general version of dispersive estimates provided in Theorem \ref{Dispersive} by using Lemma \ref{KerEst1}. In particular, for $2<q,\tilde{q}\leq \infty$, there exists a constant $C>0$ such that
\begin{equation}
\left\Vert e^{-itP}\right\Vert _{L^{\tilde{q}^{\prime}}(\mathbb{H}^{n})\rightarrow
L^{q}(\mathbb{H}^{n})}\leq C%
\begin{cases}
|t|^{-\max\left\{ \frac{1}{2}-\frac{1}{q},\,\frac{1}{2}-\frac{1}{\tilde{q}}\right\}n  } & \hbox{   if   }|t|\leq1,\\
|t|^{-\frac{3}{2}} & \hbox{   if   }|t|>1.
\end{cases}%
\end{equation}
Using these dispersive estimates one can prove the Strichartz estimates for the group $e^{-itP}$, and then show the well-posedness and scattering in the energy space $H^1(\mathbb{H}^n)\times L^2(\mathbb{H}^n)$ for GBq equation. This process has been done for the GBq equation on Euclidean space $\mathbb{R}^n\, (n \geq 3)$ (see a recent work \cite{Che2023} for more details) and also for nonlinear Schr\"odinger equations on $\mathbb{H}^n\, (n \geq 2)$ in \cite{Anker2009,Ionescu2009}. However, in the rest of this paper we are going to use Theorem \ref{Dispersive} to establish the well-posedness and scattering in Lorentz spaces of $\mathbb{H}^n$ that enable us to expand the initial data class to settings beyond the energy space (see Remark \ref{obser1} below). Moreover, in comparison to the Euclidean case, we obtain a different decay in the dispersive estimates for the Boussinesq group in $\mathbb{H}^{n}$ as well as in the scattering and asymptotic stability properties.
\end{remark}

We now extend the dispersive estimates presented in Theorem \ref{Dispersive} to the context of Lorentz spaces.

\begin{lemma}
\label{DisperInter} Let $2<p<\infty$ and $1\leq d\leq\infty$. Then, there
exists a constant $C>0$ (independent of $t$) such that
\begin{equation}
\left\Vert e^{-itP}g\right\Vert _{{(p,d)}}\leq C\phi_{p}(t)\left\Vert
g\right\Vert _{{(p^{\prime},d)}},\label{group2}%
\end{equation}
for all $g\in L^{(p^{\prime},d)}(\mathbb{H}^{n}),$ where
\begin{equation}
\phi_{p}(t)=%
\begin{cases}
|t|^{-\frac{n}{2}\left(  1-\frac{2}{p}\right)  } & \hbox{   if   }|t|\leq1,\\
|t|^{-\frac{3}{2}} & \hbox{   if   }|t|>1.
\end{cases}
\end{equation}

\end{lemma}

\textbf{Proof.} Let $p,p_{1},p_{2}\in(2,\infty)$ and $\theta\in(0,1)$ satisfy
$p^{-1}=\theta p_{1}^{-1}+(1-\theta)p_{2}^{-1}.$ Recalling that $L^{p}%
=L^{(p,p)}$, Theorem \ref{Dispersive} gives%
\begin{equation}
\left\Vert e^{-itP}g\right\Vert _{(p_{i},p_{i})}\leq C_{i}\phi_{p_{i}%
}(t)\left\Vert g\right\Vert _{(p_{i}^{\prime},p_{i}^{\prime})}%
,\label{aux-est-interp-1}%
\end{equation}
where $C_{i}>0$ is independent of $t.$ Since the functor $\left(  \cdot
,\cdot\right)  _{\theta,d}$ is of exponent $\theta$ and $L^{(p,d)}=\left(
L^{p_{1}},L^{p_{2}}\right)  _{\theta,d}$, we can interpolate estimate
(\ref{aux-est-interp-1}) to obtain
\begin{equation}
\left\Vert e^{-itP}g\right\Vert _{(p,d)}\leq K(C_{1}\phi_{p_{1}}(t))^{\theta
}(C_{2}\phi_{p_{2}}(t))^{1-\theta}\left\Vert g\right\Vert _{(p^{\prime}%
,d)},\label{aux-est-interp-2}%
\end{equation}
with $K>0$ independent of $t$. Taking $C=K\left(  C_{1}\right)  ^{\theta
}(C_{2})^{1-\theta}$ and recalling that $p^{-1}=\theta p_{1}^{-1}%
+(1-\theta)p_{2}^{-1},$ we obtain
\begin{equation}
K(C_{1}\phi_{p_{1}}(t))^{\theta}(C_{2}\phi_{p_{2}}(t))^{1-\theta}=C\phi
_{p}(t),\label{aux-est-interp-3}%
\end{equation}
and then inserting (\ref{aux-est-interp-3}) into (\ref{aux-est-interp-2})
yields the required estimate. \fin

\section{Local and global well-posedness in Lorentz spaces}\label{S4}

For $s\in\mathbb{R}$, let $\Lambda^{s}=(-\Delta_{x})^{s/2}$ and $J^{s}%
=(1-\Delta_{x})^{s/2}$ defined in $\mathbb{H}^n$, where $-\Delta_{x}$ denotes the minus Laplace-Beltrami
operator in $\mathbb{H}^{n}$. Now consider the Banach space
\[
X_{(p,d)}=L^{(p,d)}(\mathbb{H}^n)\times\Lambda^{-1}H_{(p,d)}^{-1}(\mathbb{H}^n)%
\]
equipped with the norm
\[
\left\Vert \lbrack\varphi,\psi]\right\Vert _{X_{(p,d)}}=\max\{\Vert
\varphi\Vert_{(p,d)},\Vert\psi\Vert_{\Lambda H_{(p,d)}^{-1}}\},
\]
where $H_{(p,d)}^{-1}(\mathbb{H}^n)=J(L^{(p,d)}(\mathbb{H}^n))$ and $(L^{(p,d)}(\mathbb{H}^n),\Vert\cdot\Vert_{(p,d)})$
stands for a Lorentz space (see the definition in Section \ref{S32}).

For $0<T<\infty$, let $\mathcal{L}_{\beta}^{T}$ and $\mathcal{L}_{\alpha
_{1},\alpha_{2}}$ stand for the respective Banach spaces of all Bochner
measurable pairs $[u,v]:(-T,T)\rightarrow X_{(b+1,\infty)}$ and
$[u,v]:(-\infty,\infty)\rightarrow X_{(b+1,\infty)}$ endowed with norms
\begin{equation}
\Vert\lbrack u,v]\Vert_{\mathcal{L}_{\beta}^{T}}=\sup_{0<\left\vert
t\right\vert <T}\left\vert t\right\vert ^{\beta}\left\Vert [u(t,\cdot
),v(t,\cdot)]\right\Vert _{X_{(b+1,\infty)}}\label{n1}%
\end{equation}
and
\begin{equation}
\Vert\lbrack u,v]\Vert_{\mathcal{L}_{\alpha_{1},\alpha_{2}}}=\sup
_{0<\left\vert t\right\vert <1}\left\vert t\right\vert ^{\alpha_{1}}\left\Vert
[u(t,\cdot),v(t,\cdot)]\right\Vert _{X_{(b+1,\infty)}}+\sup_{\left\vert
t\right\vert >1}\left\vert t\right\vert ^{\alpha_{2}}\left\Vert [u(t,\cdot
),v(t,\cdot)]\right\Vert _{X_{(b+1,\infty)}},\label{n2}%
\end{equation}
where $\beta=\frac{n(b-1)}{2(b+1)}$ and the exponents $\alpha_{i}$'s satisfy $\alpha_{1}=\frac{1}%
{b-1}-\frac{n}{2(b+1)}$ and $0\leq\alpha_{2}\leq\frac{3}{2}$. Moreover, define
the initial data space $\mathcal{E}_{0}$ as the set of all pairs $[u_{0}%
,v_{0}]\in\mathcal{S}^{\prime}(\mathbb{H}^{n})\times\mathcal{S}^{\prime
}(\mathbb{H}^{n})$ such that
\begin{equation}
\Vert\lbrack u_{0},v_{0}]\Vert_{\mathcal{E}_{0}}=\sup_{0<\left\vert
t\right\vert <1}\left\vert t\right\vert ^{\alpha_{1}}\left\Vert G(t)[u_{0}%
,v_{0}]\right\Vert _{X_{(b+1,\infty)}}+\sup_{\left\vert t\right\vert
>1}\left\vert t\right\vert ^{\alpha_{2}}\left\Vert G(t)[u_{0},v_{0}%
]\right\Vert _{X_{(b+1,\infty)}}<\infty.\label{n3}%
\end{equation}

Before proceeding, we define two values associated with the power $b$ of
the nonlinearity in (\ref{Bou1}) that will be useful in the statements of our
results. The first one is the positive root of the equation $nb^{2}%
-(n+2)b-2=0,$ which is given by
\begin{equation*}
b_{0}(n)=\frac{n+2+\sqrt{n^{2}+12n+4}}{2n}>1,
\end{equation*}
for $n\in\mathbb{N}$ with $n\geq2$. The second is $b_{1}(n)$ defined by $b_{1}(n)=\infty$ if $n=2$ and $b_{2}(n)=$ $\frac{n+2}{n-2}$ if $n\geq3$,
for $n\in\mathbb{N}.$

Our local-in-time well-posedness result reads as follows.

\begin{theorem}
[Local-in-time well-posedness]\label{LWP} Let $1<b<b_{0}(n)$.

\begin{enumerate}
\item[(i)] (Well-posedness) Assume that $[u_{0},v_{0}]\in L^{(\frac{b+1}%
{b},\infty)}(\mathbb{H}^n)\times L^{(\frac{b+1}{b},\infty)}(\mathbb{H}^n)$.  Then, there exsts an
existence time $T\in(0,\infty)${ such that the IVP (\ref{Bou1}) has a mild
solution $[u,v]\in$}$\mathcal{L}_{\beta}^{T}${ which is unique in a given ball
of }$\mathcal{L}_{\beta}^{T}.$ Moreover, $[u(t),v(t)]\rightharpoonup\lbrack
u_{0},v_{0}]$ in $\mathcal{S}^{\prime}(\mathbb{H}^{n})$ as $t\rightarrow0^{+}$
and the data-solution map {$[u_{0},v_{0}]\rightarrow{[u,v]}$ $\ $from
$L^{(\frac{b+1}{b},\infty)}\times$}$L^{(\frac{b+1}{b},\infty)}${\ to
}$\mathcal{L}_{\beta}^{T}$ is locally Lipschitz{. }

\item[(ii)] ($L^{(b+1,d)}$-Regularity) Let {$[u_{0},v_{0}]\in L^{(\frac
{b+1}{b},d)}(\mathbb{H}^n)\times$}$L^{(\frac{b+1}{b},d)}(\mathbb{H}^n)$ with $1\leq d\leq\infty.$ Then,
the previous solution satisfies
\begin{equation}
\sup_{0<\left\vert t\right\vert <T}\left\vert t\right\vert ^{\beta}\Vert
u(t,\cdot)\Vert_{(b+1,d)}<\infty\text{ and }\sup_{0<\left\vert t\right\vert
<T}\left\vert t\right\vert ^{\beta}\Vert v(t,\cdot)\Vert_{\Lambda
^{-1}H_{(b+1,d)}^{-1}}<\infty.\label{reg1}%
\end{equation}

\end{enumerate}
\end{theorem}

\begin{remark}
\label{obser1}

\begin{itemize}
\item[(i)] (Infinite energy data) Using the initial class {$L^{(\frac{b+1}%
{b},\infty)}(\mathbb{H}^n)\times$}$L^{(\frac{b+1}{b},\infty)}(\mathbb{H}^n)$ given in Theorem \ref{LWP},
we are able to cover initial data \ $[u_{0},v_{0}]$ such that $u_{0}%
,v_{0}\notin L^{2}(\mathbb{H}^{n})$. For that, let $\{x^{(j)}\}_{j\in
\mathbb{N}}$ and $\{\tilde{x}^{(j)}\}_{j\in\mathbb{N}}$ be two sets of fixed
points in $\mathbb{H}^{n}$ and consider
\begin{equation}
u_{0}=\Sigma_{j=1}^{\infty}\kappa_{j}R_{k,j}(d(x,x^{(j)}))\Gamma_{j}(x)\text{
and }v_{0}=\Sigma_{j=1}^{\infty}\mu_{j}S_{k,j}(d(x,\tilde{x}^{(j)}%
)\tilde{\Gamma}_{j}(x)\label{data}%
\end{equation}
where
\begin{equation}
\Gamma_{j}(x)=\left\{
\begin{array}
[c]{c}%
d(x,x^{(j)})^{-\frac{nb}{b+1}-k}\text{, if }0<d(x,x^{(j)})<1,\\
e^{-(\frac{(n-1)b}{b+1}+s)d(x,x^{(j)})}\,\text{, if }d(x,x^{(j)})\geq1,
\end{array}
\right.  \text{ and }\tilde{\Gamma}_{j}(x)=\left\{
\begin{array}
[c]{c}%
d(x,\tilde{x}^{(j)})^{-\frac{nb}{b+1}-\tilde{k}}\text{, if }0<d(x,\tilde
{x}^{(j)})<1,\\
e^{-(\frac{(n-1)b}{b+1}+\tilde{s})d(x,\tilde{x}^{(j)})}\,\text{, if
}d(x,\tilde{x}^{(j)})\geq1,
\end{array}
\right.  \label{data-2}%
\end{equation}
where $\kappa_{j},\mu_{j}$ are constant and $R_{k,j}(z),S_{k,j}(z):\mathbb{R}%
_{+}\rightarrow\mathbb{R}$ are continuous functions satisfying
\begin{align*}
\left\vert R_{k,j}(z)\right\vert  &  \lesssim\left\vert z\right\vert
^{k}\text{ and }\left\vert S_{k,j}(z)\right\vert \lesssim\left\vert
z\right\vert ^{\tilde{k}},\text{ as }z\rightarrow0,\text{ }\\
\text{ \ \ \ \ }\left\vert R_{k,j}(z)\right\vert  &  \lesssim e^{sz}\text{ and
\ \ }\left\vert S_{k,j}(z)\right\vert \lesssim e^{\tilde{s}z},\text{ as
}z\rightarrow\infty,
\end{align*}
where $k,\tilde{k}$ and $s,\tilde{s}$ are nonnegative real numbers. In fact,
the data (\ref{data}) have not even finite local-energy, that is,\ $u_{0}%
\notin L_{loc}^{2}(\mathbb{H}^{n})$ and $v_{0}\notin L_{loc}^{2}%
(\mathbb{H}^{n})$. Furthermore, we have that $\phi=\Delta(v_{0})\notin
H^{-2}(\mathbb{H}^{n}).$\bigskip

\item[(ii)] With a slight adaptation of its statement and proof, Theorem
\ref{LWP} still works well by replacing the initial data class {$L^{(\frac
{b+1}{b},\infty)}(\mathbb{H}^n)\times$}$L^{(\frac{b+1}{b},\infty)}(\mathbb{H}^n)$ {\ by }the space of all
pairs $[u_{0},v_{0}]\in\mathcal{S}^{\prime}(\mathbb{H}^{n})\times
\mathcal{S}^{\prime}(\mathbb{H}^{n})$ such that%
\begin{equation}
\sup_{0<\left\vert t\right\vert <T}\left\vert t\right\vert ^{\beta}\Vert
G(t)[u_{0},v_{0}]\Vert_{X_{(b+1,\infty)}}<\infty\text{ \ (see (\ref{aux21})).}\label{aux22}%
\end{equation}

\end{itemize}
\end{remark}

In what follows, we state our result on global well-posedness.

\begin{theorem}
[Global well-posedness]\label{GWP}Let $b_{0}(n)<b<b_{1}(n)$,{ }$\alpha
_{1}=\frac{1}{b-1}-\frac{n}{2(b+1)},$ $0\leq\alpha_{2}\leq\frac{3}{2}$ and
suppose that $[u_{0},v_{0}]\in\mathcal{E}_{0}.$

\begin{enumerate}
\item[(i)] (Well-posedness) There exists $\varepsilon>0$ such that {the IVP
(\ref{Bou1}) has a unique global-in-time mild solution $[u,v]\in$}%
$\mathcal{L}_{\alpha_{1},\alpha_{2}}${ satisfying $\Vert\lbrack u,v]\Vert
_{\mathcal{L}_{\alpha_{1},\alpha_{2}}}\leq2\varepsilon$,} {provided that
}$\Vert\lbrack u_{0},v_{0}]\Vert_{\mathcal{E}_{0}}\leq\varepsilon$. Moreover,
$[u(t),v(t)]\rightharpoonup\lbrack u_{0},v_{0}]$ in $\mathcal{S}^{\prime
}(\mathbb{H}^{n})$ as $t\rightarrow0^{+}$ and the data-solution map is locally Lipschitz.

\item[(ii)] ($L^{(b+1,d)}$-regularity) {Let }$1\leq d\leq\infty,$ {$0\leq
h<1-{\alpha}_{1}{b,}$ }$h\leq\frac{3}{2}-\alpha_{2},$ and {assume further that
}%
\begin{equation}
\sup_{\left\vert t\right\vert <1}\left\vert t\right\vert ^{\alpha_{1}+h}\Vert
G(t)[u_{0},v_{0}]\Vert_{X_{(b+1,d)}}+\sup_{\left\vert t\right\vert
>1}\left\vert t\right\vert ^{\alpha_{2}+h}\Vert G(t)[u_{0},v_{0}%
]\Vert_{X_{(b+1,d)}}<\infty.\label{conreg2}%
\end{equation}
Then the solution $[u,v]$ obtained in item (i) satisfies
\begin{align}
\sup_{\left\vert t\right\vert <1}\left\vert t\right\vert ^{\alpha_{1}+h}\Vert
u(t,\cdot)\Vert_{(b+1,d)} &  <\infty,\text{ }\sup_{\left\vert t\right\vert
>1}\left\vert t\right\vert ^{\alpha_{2}+h}\Vert u(t,\cdot)\Vert_{(b+1,d)}%
<\infty,\label{reg-2-1}\\
\sup_{\left\vert t\right\vert <1}\left\vert t\right\vert ^{\alpha_{1}+h}\Vert
v(t,\cdot)\Vert_{\Lambda^{-1}H_{(b+1,d)}^{-1}} &  <\infty,\text{ and }%
\sup_{\left\vert t\right\vert >1}\left\vert t\right\vert ^{\alpha_{2}+h}\Vert
v(t,\cdot)\Vert_{\Lambda^{-1}H_{(b+1,d)}^{-1}}<\infty,\label{reg-2-2}%
\end{align}
provided that $\Vert\lbrack u_{0},v_{0}]\Vert_{\mathcal{E}_{0}}\leq
\varepsilon_{0}$, for some $0<\varepsilon_{0}\leq\varepsilon$.
\end{enumerate}
\end{theorem}

\

\begin{remark}
\label{obser222} Parameters $b_{0}(n)$ and $b_{1}(n)$ appear from time
integrability conditions close to zero in the estimates of the nonlinear
components of the integral formulation (\ref{intergralEq}). In fact, they seem to be of
universal characteristic. For example, in some works on Euclidean case, they
arise in connection with critical Sobolev index and optimal time-decay for the
norm $\Vert\cdot\Vert_{L^{b+1}}$, respectively, see e.g. \cite{Kato, Naka-Oza}.
\end{remark}

\

\subsection{Boussinesq group estimates in $X_{(p,d)}$-spaces}\label{S41}

In this section, we provide estimates for the Boussinesq group $G(t)$ defined
in (\ref{aux-Bou-rel-1}) in the $X_{(p,d)}$-setting.

\begin{lemma}
\label{estim1} Let $n\geq2$, $1\leq d\leq\infty$, and $2<p<\infty$. Then,
there exists a constant $C_{G}>0$ (independent of $t$) such that
\begin{equation}
\left\Vert G(t)[\varphi,\psi]\right\Vert _{X_{(p,d)}}\leq C_{G}\phi_{p}%
(t)\max\{\left\Vert \varphi\right\Vert _{{(p}^{\prime}{,d)}},\left\Vert
\psi\right\Vert _{{(p}^{\prime}{,d)}}\},\label{semigroup}%
\end{equation}
for all $\varphi,\psi\in$ $L^{(p^{\prime},d)}(\mathbb{H}^{n}),$ where
\begin{equation}
\phi_{p}(t)=%
\begin{cases}
|t|^{-\frac{n}{2}\left(  1-\frac{2}{p}\right)  } & \hbox{   if   }|t|\leq1,\\
|t|^{-\frac{3}{2}} & \hbox{   if   }|t|>1.
\end{cases}
\label{aux-phi-2}%
\end{equation}

\end{lemma}

\textbf{Proof.} First recall the operators $l_{1}(t)=g_{1}(t)$ and $l_{2}(t)$
defined in (\ref{aux-g1}) and (\ref{aux-g2}), respectively. Using estimate
(\ref{group2}), we obtain that%

\begin{equation}
\Vert l_{i}(t)\varphi\Vert_{{(p,d)}}\leq C\phi_{p}(t)\Vert\varphi\Vert
_{{(p}^{\prime}{,d)}}.\label{estimaux}%
\end{equation}
Denoting by $\pi_{i}$ the $i$-th projection, we can write $G(t)\overrightarrow
{z}=(\pi_{1}G(t)\overrightarrow{z},\pi_{2}G(t)\overrightarrow{z})$ with
$\overrightarrow{z}=[\varphi,\psi]$. Noting that $g_{2}(t)=\Lambda J^{-1}%
l_{2}(t)$ and using the $L^{(p^{\prime},d)}$-continuity of the operator
$\Lambda J^{-1},$ it follows from (\ref{BouOP}) and (\ref{estimaux}) that
\begin{align}
\left\Vert \pi_{1}G(t)\overrightarrow{z}\right\Vert _{{(p,d)}} &
\leq\left\Vert g_{1}(t)\varphi\right\Vert _{(p,d)}+\left\Vert g_{2}%
(t)\psi\right\Vert _{(p,d)}\nonumber\\
&  \leq\left\Vert l_{1}(t)\varphi\right\Vert _{(p,d)}+\left\Vert \Lambda
J^{-1}l_{2}(t)\psi\right\Vert _{(p,d)}\nonumber\\
&  \leq C\phi_{p}(t)\max\{\left\Vert \varphi\right\Vert _{{(p}^{\prime}{,d)}%
},\left\Vert \psi\right\Vert _{{(p}^{\prime}{,d)}}\}\text{,}\label{est2}%
\end{align}
for all $t\neq0,$where the constant $C>0$ is independent of $t$ (in fact, it depends solely on $n,p$). Moreover, since $\Lambda J^{-1}g_{3}(t)=l_{2}(t)$ (see
(\ref{BouOP})), we can express
\[
\Lambda J^{-1}\pi_{2}G(t)\overrightarrow{z}=\Lambda J^{-1}g_{3}(t)\varphi
+\Lambda J^{-1}g_{4}(t)\psi=l_{2}(t)\varphi+\Lambda J^{-1}g_{1}(t)\psi.
\]
Thus, the $L^{(p^{\prime},d)}$-continuity of $\Lambda J^{-1}$ and estimate
(\ref{estimaux}) lead us to
\begin{align}
\left\Vert \pi_{2}G(t)\overrightarrow{z}\right\Vert _{\Lambda^{-1}%
H_{(p,d)}^{-1}} &  =\left\Vert \Lambda\pi_{2}G(t)\overrightarrow{z}\right\Vert
_{H_{(p,d)}^{-1}}\nonumber\\
&  =\left\Vert \Lambda J^{-1}\pi_{2}G(t)\overrightarrow{z}\right\Vert
_{(p,d)}\nonumber\\
&  \leq\left\Vert l_{2}(t)\varphi\right\Vert _{(p,d)}+\left\Vert \Lambda
J^{-1}g_{1}(t)\psi\right\Vert _{(p,d)}\nonumber\\
&  \leq C\phi_{p}(t)\max\{\left\Vert \varphi\right\Vert _{{(p}^{\prime}{,d)}%
},\left\Vert \psi\right\Vert _{{(p}^{\prime}{,d)}}\}.\text{ }\label{est222}%
\end{align}
Estimate (\ref{semigroup}) follows by putting together (\ref{est2}) and
(\ref{est222}). \fin

\subsection{Nonlinear estimates}

\label{S42}

In order to carry out a contraction argument for the integral equation
(\ref{intergralEq}), we need to develop suitable local and global-in-time
estimates for its nonlinear part, namely
\begin{equation}\label{aux-non-1}
\mathcal{G}(u)=-\int_{0}^{t}G(t-s)[0,f(u(s))]ds.
\end{equation}
Before stating such estimates, let us recall the so-called Beta function
\begin{equation}
B(\theta_{1},\theta_{2})=\int_{0}^{1}(1-s)^{\theta_{1}-1}s^{\theta_{2}%
-1}ds,\label{aux-beta-1}%
\end{equation}
which is finite for $\theta_{1},\theta_{2}>0.$ Furthermore, considering
$l_{i}=1-\theta_{i},$ $i=1,2,$ we can make the change of variable
$s\rightarrow st$ to express
\begin{equation}
\int_{0}^{t}(t-s)^{-l_{1}}s^{-l_{2}}ds=t^{1-l_{1}-l_{2}}\int_{0}%
^{1}(1-s)^{-l_{1}}s^{-l_{2}}ds=t^{1-l_{1}-l_{2}}B(1-l_{1},1-l_{2}%
),\label{Beta}%
\end{equation}
for each fixed $t>0.$

Our nonlinear estimates are the subject of the next proposition.

\begin{proposition}
[Nonlinear estimates]\label{lem8}

\begin{itemize}
\item[(i)] (Local-in-time estimates) Let $1<b<b_{0}(n),$ $1\leq d\leq\infty,$
and $\beta=\frac{n(b-1)}{2(b+1)}$. Then, there is a universal constant
$K_{\beta}>0$ such that%
\begin{align}
\sup_{\left\vert t\right\vert <T}\left\vert t\right\vert ^{\beta}%
\Vert\mathcal{G}(u)-\mathcal{G}(\tilde{u})\Vert_{X_{(b+1,d)}} &  \leq
K_{\beta}T^{1-b\beta}\sup_{0<\left\vert t\right\vert <T}\left\vert
t\right\vert ^{\beta}\Vert u-\tilde{u}\Vert_{(b+1,d)}\label{sc9}\\
&  \times\sup_{0<\left\vert t\right\vert <T}(\left\vert t\right\vert
^{\beta(b-1)}\Vert u\Vert_{(b+1{,\infty)}}^{b-1}+\left\vert t\right\vert
^{\beta(b-1)}\Vert\tilde{u}\Vert_{(b+1{,\infty)}}^{b-1}).\nonumber\\
& \nonumber
\end{align}

\item[(ii)] (Global-in-time estimates) Let $b_{0}(n)<b<b_{1}(n),$ $1\leq
d\leq\infty$, $\alpha_{1}=\frac{1}{b-1}-\frac{n}{2(b+1)},$ {$0\leq
h<1-{\alpha}_{1}b$ and }$0\leq\alpha_{2}\leq\frac{3}{2}-h$. Then, there is a
universal constant $K_{\alpha_{1},\alpha_{2},h}>0$ such that
\begin{align}
&  \sup_{0<\left\vert t\right\vert <1}\left\vert t\right\vert ^{\alpha_{1}%
+h}\Vert\mathcal{G}(u)-\mathcal{G}(\tilde{u})\Vert_{X_{(b+1,d)}}%
+\sup_{\left\vert t\right\vert >1}\left\vert t\right\vert ^{\alpha_{2}+h}%
\Vert\mathcal{G}(u)-\mathcal{G}(\tilde{u})\Vert_{X_{(b+1,d)}}\nonumber\\
&  \leq K_{\alpha_{1},\alpha_{2},h}\left(  \sup_{0<\left\vert t\right\vert
<1}\left\vert t\right\vert ^{\alpha_{1}+h}\Vert u-\tilde{u}\Vert
_{(b+1,d)}+\sup_{\left\vert t\right\vert >1}\left\vert t\right\vert
^{\alpha_{2}+h}\Vert u-\tilde{u}\Vert_{(b+1,d)}\right)  \label{sc8}\\
&  \times\left(  \sup_{0<\left\vert t\right\vert <1}\left\vert t\right\vert
^{\alpha_{1}(b-1)}\left(  \Vert u\Vert_{(b+1,\infty)}^{b-1}+\Vert\tilde
{u}\Vert_{(b+1,\infty)}^{b-1}\right)  +\sup_{\left\vert t\right\vert
>1}\left\vert t\right\vert ^{\alpha_{2}(b-1)}\left(  \Vert u\Vert
_{(b+1,\infty)}^{b-1}+\Vert\tilde{u}\Vert_{(b+1,\infty)}^{b-1}\right)
\right)  .\nonumber\\
& \nonumber
\end{align}

\end{itemize}
\end{proposition}

\textbf{Proof}{\textit{.}} We are going to prove the estimates only for $t>0$,
since the otherwise case follows an entirely parallel way. Also, for item
(i)$,$ we can assume $T<1$ (without loss of generality), because $t^{a_{1}%
}\sim t^{a_{2}}$ when $t\in(1,T)$ for each fixed pair $a_{1},a_{2}%
\in\mathbb{R}$.

By employing Lemma \ref{estim1} with $p=\frac{b{+1}}{b}$ and Holder inequality
(\ref{Holder}), we can estimate
\begin{align}
\Vert\mathcal{G}(u)-\mathcal{G}(\tilde{u})\Vert_{X_{(b+1,d)}} &  \leq\int
_{0}^{t}\Vert G(t-s)[0,f(u)-f(\tilde{u})]\Vert_{X_{(b+1,d)}}ds\nonumber\\
&  \hspace{-3.3cm}\leq C\int_{0}^{t}\phi_{b+1}(t-s)\Vert(\left\vert
u-\tilde{u}\right\vert )(\left\vert u\right\vert ^{b-1}+\left\vert \tilde
{u}\right\vert ^{b-1})\Vert_{(\frac{b{+1}}{b}{,d)}}ds\nonumber\\
&  \hspace{-3.3cm}\leq C\int_{0}^{t}(t-s)^{-\beta}\Vert u-\tilde{u}%
\Vert_{(b+1{,d)}}\left(  \Vert u\Vert_{(b+1{,\infty)}}^{\rho-1}+\Vert\tilde
{u}\Vert_{(b+1{,\infty)}}^{\rho-1}\right)  ds,\label{aux2}%
\end{align}
for $t\in(0,T).$ Note that $\beta b<1$ (and $\beta<1$) when $1<b<b_{0}(n).$
Now, taking the corresponding time-weighted supremum over $t\in(0,T)$ and
using (\ref{Beta}), we can bound the r.h.s of (\ref{aux2}) as follows
\begin{align*}
&  C\int_{0}^{t}(t-s)^{-\beta}s^{-\beta b}ds\left(  \sup_{0<t<T}t^{\beta}\Vert
u-\tilde{u}\Vert_{(b+1{,d)}}\sup_{0<t<T}\left(  t^{\beta(b-1)}\Vert
u\Vert_{(b+1{,d)}}^{b-1}+t^{\beta(b-1)}\Vert\tilde{u}\Vert_{(b+1{,d)}}%
^{b-1}\right)  \right)  \\
&  =K_{\beta}t^{-\beta}t^{1-\beta b}\left(  \sup_{0<t<T}t^{\beta}\Vert
u-\tilde{u}\Vert_{(b+1{,d)}}\sup_{0<t<T}\left(  t^{\beta(b-1)}\Vert
u\Vert_{(b+1{,d)}}^{b-1}+t^{\beta(b-1)}\Vert\tilde{u}\Vert_{(b+1{,d)}}%
^{b-1}\right)  \right)  ,
\end{align*}
which leads us to (\ref{sc9}), where $K_\beta = CB(1-\beta,1-\beta b)$.

Next we turn to item (ii). Applying Lemma \ref{estim1} yields
\begin{align}
\Vert\mathcal{G}(u)-\mathcal{G}(\tilde{u})\Vert_{X_{(b+1,d)}} &  \leq
C\int_{0}^{t}\phi_{b+1}(t-s)\Vert(\left\vert u-\tilde{u}\right\vert
)(\left\vert u\right\vert ^{b-1}+\left\vert \tilde{u}\right\vert ^{b-1}%
)\Vert_{(\frac{b{+1}}{b}{,d)}}ds\nonumber\\
&  =C\int_{0}^{\delta_{t}}\phi_{b+1}(t-s)\Vert(\left\vert u-\tilde
{u}\right\vert )(\left\vert u\right\vert ^{b-1}+\left\vert \tilde
{u}\right\vert ^{b-1})\Vert_{(\frac{b{+1}}{b}{,d)}}ds\nonumber\\
&  +C\int_{\delta_{t}}^{t}\phi_{b+1}(t-s)\Vert(\left\vert u-\tilde
{u}\right\vert )(\left\vert u\right\vert ^{b-1}+\left\vert \tilde
{u}\right\vert ^{b-1})\Vert_{(\frac{b{+1}}{b}{,d)}}ds,\label{aux-est-g1}%
\end{align}
where $\delta_{t}\in\lbrack0,t).$ Let us first consider $t\in(0,2).$ For fixed
$a_{1},a_{2}\in\mathbb{R}$, note that%
\begin{equation}
t^{a_{1}}\sim t^{a_{2}}\text{ when }t\in(1,2).\label{aux-t-power-1}%
\end{equation}
Note also that $t-s\in(1,2)$ provided that $t\in(1,2)$ and $s\in(0,t-1).$

In the next estimate, we consider $\delta_{t}=0$ when $t\in(0,1)$ and
$\delta_{t}=t-1$ when $t\in(1,2)$. Then, using the definition of $\phi_{p}$
(see (\ref{aux-phi-2})), (\ref{aux-t-power-1}), H\"{o}lder inequality
(\ref{Holder}) and property (\ref{Beta}), we arrive at
\begin{align}
\text{R.H.S. of (\ref{aux-est-g1})} &  \leq C\int_{0}^{\delta_{t}%
}(t-s)^{-\beta}\Vert(\left\vert u-\tilde{u}\right\vert )(\left\vert
u\right\vert ^{b-1}+\left\vert \tilde{u}\right\vert ^{b-1})\Vert_{(\frac
{b{+1}}{b}{,d)}}ds\nonumber\\
&  +C\int_{\delta_{t}}^{t}(t-s)^{-\beta}\Vert(\left\vert u-\tilde
{u}\right\vert )(\left\vert u\right\vert ^{b-1}+\left\vert \tilde
{u}\right\vert ^{b-1})\Vert_{(\frac{b{+1}}{b}{,d)}}ds\nonumber\\
&  =C\int_{0}^{t}(t-s)^{-\beta}\Vert(\left\vert u-\tilde{u}\right\vert
)(\left\vert u\right\vert ^{b-1}+\left\vert \tilde{u}\right\vert ^{b-1}%
)\Vert_{(\frac{b{+1}}{b}{,d)}}ds\nonumber\\
&  \leq C\int_{0}^{t}(t-s)^{-\beta}s^{-\alpha_{1}b-h}ds\nonumber\\
&  \times\left(  \sup_{0<t<1}t^{\alpha_{1}+h}\Vert u-\tilde{u}\Vert
_{(b+1{,d)}}\sup_{0<t<1}\left(  t^{\alpha_{1}(b-1)}\Vert u\Vert_{(b+1{,\infty
)}}^{b-1}+t^{\alpha_{1}(b-1)}\Vert\tilde{u}\Vert_{(b+1{,\infty)}}%
^{b-1}\right)  \right)  \nonumber\\
&  =K_{\alpha_{1},\alpha_{2},h}t^{-\alpha_{1}-h}\left(  \sup_{t>0}%
t^{\alpha_{1}+h}\Vert u-\tilde{u}\Vert_{(b+1{,d)}}\sup_{t>0}\left(
t^{\alpha_{1}(b-1)}\Vert u\Vert_{(b+1{,\infty)}}^{b-1}+t^{\alpha_{1}%
(b-1)}\Vert\tilde{u}\Vert_{(b+1{,\infty)}}^{b-1}\right)\right),\label{aux-est-g2}
\end{align}
because $\beta<1$, $\alpha_{1}b+h<1$ and $\beta+(b-1)\alpha_{1}=1$, where $K_{\alpha_1,\alpha_2,h} = CB(1-\beta,1-\alpha_1 b-h)$.

In the sequel, let $t\geq2$ and consider $\delta_{t}=1$ in (\ref{aux-est-g1}).
It follows that
\[
\Vert\mathcal{G}(u)-\mathcal{G}(\tilde{u})\Vert_{X_{(b+1,d)}}\leq A_{1}+A_{2},
\]
where
\begin{align*}
A_{1} &  \leq C\int_{0}^{1}\phi_{b+1}(t-s)s^{-\alpha_{1}b-h}ds\left(
\sup_{0<t<1}t^{\alpha_{1}+h}\Vert u-\tilde{u}\Vert_{(b+1{,d)}}\sup
_{0<t<1}\left(  t^{\alpha_{1}(b-1)}\Vert u\Vert_{(b+1{,\infty)}}%
^{b-1}+t^{\alpha_{1}(b-1)}\Vert\tilde{u}\Vert_{(b+1{,\infty)}}^{b-1}\right)
\right)  ,\\
A_{2} &  \leq C\int_{1}^{t}\phi_{b+1}(t-s)s^{-\alpha_{2}b-h}ds\left(
\sup_{t>1}t^{\alpha_{2}+h}\Vert u-\tilde{u}\Vert_{(b+1{,d)}}\sup_{t>1}\left(
t^{\alpha_{2}(b-1)}\Vert u\Vert_{(b+1{,\infty)}}^{b-1}+t^{\alpha_{2}%
(b-1)}\Vert\tilde{u}\Vert_{(b+1{,\infty)}}^{b-1}\right)  \right)  .
\end{align*}
Therefore, we can estimate
\begin{align}
&  t^{\alpha_{2}+h}\Vert\mathcal{G}(u)-\mathcal{G}(\tilde{u})\Vert
_{X_{(b+1,d)}}\nonumber\\
&  \leq CL(t)\left(  \sup_{0<t<1}t^{\alpha_{1}+h}\Vert u-\tilde{u}%
\Vert_{(b+1,d)}+\sup_{t>1}t^{\alpha_{2}+h}\Vert u-\tilde{u}\Vert
_{(b+1,d)}\right)  \nonumber\\
&  \times\left(  \sup_{0<t<1}t^{\alpha_{1}(b-1)}\left(  \Vert u\Vert
_{(b+1,\infty)}^{b-1}+\Vert\tilde{u}\Vert_{(b+1,\infty)}^{b-1}\right)
+\sup_{t>1}t^{\alpha_{2}(b-1)}\left(  \Vert u\Vert_{(b+1,\infty)}^{b-1}%
+\Vert\tilde{u}\Vert_{(b+1,\infty)}^{b-1}\right)  \right)  ,\label{aux-est-g3}%
\end{align}
where
\begin{equation}
L(t)=Ct^{\alpha_{2}+h}\left(  \int_{0}^{1}\phi_{b+1}(t-s)s^{-\alpha_{1}%
b-h}ds+\int_{1}^{t}\phi_{b+1}(t-s)s^{-\alpha_{2}b-h}ds\right)
.\label{aux-est-g4}%
\end{equation}
By (\ref{aux-phi-2}) and using $\alpha_{1}b+h<1$, $\beta=\frac{n}{2}%
(\frac{b-1}{b+1})<1$ and $\alpha_{2}+h\leq\frac{3}{2},$ we obtain that%

\begin{align}
L(t) &  =Ct^{\alpha_{2}+h}\left(  \int_{0}^{1}\phi_{b+1}(t-s)s^{-\alpha
_{1}b-h}ds+\int_{1}^{t-1}\phi_{b+1}(t-s)s^{-\alpha_{2}b-h}ds+\int_{t-1}%
^{t}\phi_{b+1}(t-s)s^{-\alpha_{2}b-h}ds\right)  \nonumber\\
&  \leq Ct^{\alpha_{2}+h}\left(  \int_{0}^{1}(t-s)^{-\frac{3}{2}}%
s^{-\alpha_{1}b-h}ds+\int_{1}^{t/2}(t-s)^{-\frac{3}{2}}s^{-\alpha_{2}%
b-h}ds\right.  \nonumber\\
&  \left.  \text{ \ \ \ \ \ \ \ \ \ \ \ \ \ \ }+\int_{t/2}^{t-1}%
(t-s)^{-\frac{3}{2}}s^{-\alpha_{2}b-h}ds+\int_{t-1}^{t}(t-s)^{-\frac{n}%
{2}(\frac{b-1}{b+1})}s^{-\alpha_{2}b-h}ds\right)  \nonumber\\
&  \leq Ct^{\alpha_{2}+h}\left(  (t-1)^{-\frac{3}{2}}\int_{0}^{1}%
s^{-\alpha_{1}b-h}ds+(t/2)^{-\frac{3}{2}}\int_{1}^{t/2}s^{-\alpha_{2}%
b-h}ds\right.  \nonumber\\
&  \left.  \text{\ \ \ \ \ \ \ \ \ \ \ \ \ \ }+(t/2)^{-\alpha_{2}b-h}%
\int_{t/2}^{t-1}(t-s)^{-\frac{3}{2}}ds+(t-1)^{-\alpha_{2}b-h}\int_{t-1}%
^{t}(t-s)^{-\beta}ds\right)  \nonumber\\
&  \leq Ct^{\alpha_{2}+h}\left(  (t-1)^{-\frac{3}{2}}(1)^{1-\alpha_{1}%
b-h}+(t/2)^{-\frac{3}{2}}\left\vert (t/2)^{1-\alpha_{2}b-h}-1\right\vert
\right.  \text{ (if }\alpha_{2}b+h\neq1\text{)}\nonumber\\
&  \left.  \text{\ \ \ \ \ \ \ \ \ \ \ \ \ \ }+(t/2)^{-\alpha_{2}%
b-h}\left\vert (1)^{-\frac{1}{2}}-(t/2)^{-\frac{1}{2}}\right\vert
+(t-1)^{-\alpha_{2}b-h}(1)^{1-\beta}\right)  \nonumber\\
&  \leq Ct^{\alpha_{2}+h}\left(  t^{-\frac{3}{2}}+t^{-\alpha_{2}b-h-\frac
{1}{2}}+t^{-\frac{3}{2}}+t^{-\alpha_{2}b-h}+t^{-\alpha_{2}b-h-\frac{1}{2}%
}+t^{-\alpha_{2}b-h}\right)  \nonumber\\
&  \leq C\left(  t^{\alpha_{2}+h-\frac{3}{2}}+t^{-\alpha_{2}(b-1)-\frac{1}{2}%
}+t^{-\alpha_{2}(b-1)}\right)  \nonumber\\
&  \leq C,\text{ for }t\geq2.\label{aux-est-g5}%
\end{align}
In the case $\alpha_{2}b+h=1$, we also arrive at the same estimate $L(t)\leq
C$, since $(t)^{-\frac{3}{2}}\int_{1}^{t/2}s^{-1}ds=t^{-\frac{3}{2}}%
\ln(t/2)\lesssim t^{-\frac{3}{2}+\varepsilon}$ with small $\varepsilon>0$, for
$t\geq2.$

Estimate (\ref{sc8}) follows by considering (\ref{aux-est-g2}) in
(\ref{aux-est-g1}) and (\ref{aux-est-g5}) in (\ref{aux-est-g3}) and then
taking the supremum over $t\in(0,2)$ and $t\in\lbrack2,\infty),$ respectively,
and using (\ref{aux-t-power-1}). \fin

\subsection{Proof of the local well-posedness result}

\label{S43}

\smallskip\textbf{\noindent Local-in-time well-posedness.} Let us consider the
closed ball $\mathcal{H}_{2\eta}=\{[u,v]\in\mathcal{L}_{\beta}^{T}%
;\Vert\lbrack u,v]\Vert_{\mathcal{L}_{\beta}^{T}}\leq2\eta\}$ equipped with
the complete metric $\mathcal{M}(\cdot,\cdot)$ given by
\[
\mathcal{M}([u,v],[\tilde{u},\tilde{v}])=\Vert\lbrack u-\tilde{u},v-\tilde
{v}]\Vert_{\mathcal{L}_{\beta}^{T}}.
\]
We are going to show that the map
\begin{equation}\label{fi2-1}
\Phi([u,v])=G(t)[u_{0},v_{0}]+\mathcal{G}(u)
\end{equation}
is a contraction on $(\mathcal{H}_{2\eta},\mathcal{M}),$ for some $\eta>0$
chosen later, where $\mathcal{G}(u)$ is as in (\ref{aux-non-1}). Applying
Lemma \ref{estim1} with $p=b+1$ and $d=\infty$, and using that $t^{a_{1}}\sim
t^{a_{2}}$ in $(1,t_{0})$ for fixed $a_{1},a_{2}\in\mathbb{R}$ and $t_{0}>1$,
we can estimate%
\begin{align}
\Vert G(t)[u_{0},v_{0}]\Vert_{\mathcal{L}_{\beta}^{T}} &  =\sup_{0<\left\vert
t\right\vert <T}\left\vert t\right\vert ^{\beta}\Vert G(t)[u_{0},v_{0}%
]\Vert_{X_{(b+1,\infty)}}\nonumber\\
&  \leq C_{G}\left(  \sup_{0<\left\vert t\right\vert <T}\left\vert
t\right\vert ^{\beta}\phi_{b+1}(t)\right)  \Vert\lbrack u_{0},v_{0}%
]\Vert_{(\frac{b+1}{b},\infty)}\nonumber\\
&  \leq C_{1}\Vert\lbrack u_{0},v_{0}]\Vert_{(\frac{b+1}{b},\infty
)}.\label{aux21}%
\end{align}
Next, taking $\eta=C_{1}\Vert\lbrack u_{0},v_{0}]\Vert_{(\frac{b+1}{b}%
,\infty)}$ and using Proposition \ref{lem8} (i) with $\tilde{u}=0$ and
$d=\infty$ lead us to%

\begin{align*}
\Vert\Phi([u,v])\Vert_{\mathcal{L}_{\beta}^{T}} &  \leq\Vert G(t)[u_{0}%
,v_{0}]\Vert_{\mathcal{L}_{\beta}^{T}}+\Vert\mathcal{G}(u)\Vert_{\mathcal{L}%
_{\beta}^{T}}\\
&  \leq C_{1}\Vert\lbrack u_{0},v_{0}]\Vert_{(\frac{b+1}{b},\infty)}+K_{\beta
}T^{1-b\beta}\Vert\lbrack u,v]\Vert_{\mathcal{L}_{\beta}^{T}}^{b}\\
&  \leq\eta+K_{\beta}T^{1-b\beta}2^{b}\eta^{b}\leq2\eta,
\end{align*}
for all $[u,v]\in\mathcal{H}_{2\eta},$ provided that $T>0$ satisfies
$2^{b-1}K_{\beta}\eta^{b-1}T^{1-b\beta}<\frac{1}{2}.$ Therefore,
$\Phi(\mathcal{H}_{2\eta})\subset\mathcal{H}_{2\eta}.$ Moreover, by
Proposition \ref{lem8} (i), we\ have that%

\begin{align}
\Vert\Phi([u,v])-\Phi([\tilde{u},\tilde{v}])\Vert_{\mathcal{L}_{\beta}^{T}} &
=\Vert\mathcal{G}(u)-\mathcal{G}(\tilde{u})\Vert_{\mathcal{L}_{\beta}^{T}%
}\nonumber\\
&  \leq K_{\beta}T^{1-b\beta}\Vert\lbrack u,v]-[\tilde{u},\tilde{v}%
]\Vert_{\mathcal{L}_{\beta}^{T}}\left(  \Vert\lbrack u,v]\Vert_{\mathcal{L}%
_{\beta}^{T}}^{b-1}+\Vert\lbrack\tilde{u},\tilde{v}]\Vert_{\mathcal{L}_{\beta
}^{T}}^{b-1}\right)  \nonumber\\
&  \leq K_{\beta}T^{1-b\beta}2^{b}\eta^{b-1}\Vert\lbrack u,v]-[\tilde
{u},\tilde{v}]\Vert_{\mathcal{L}_{\beta}^{T}},\label{aux3}%
\end{align}
for all $[u,v],[\tilde{u},\tilde{v}]\in$ $\mathcal{H}_{2\eta}$. It follows
that $\Phi$ is a contraction in $\mathcal{H}_{2\eta}$, as $K_{\beta
}T^{1-b\beta}2^{b}\eta^{b-1}<1$. By the contraction mapping principle, the map
$\Phi$ has a fixed point which is a solution $[u,v]$ in $\mathcal{L}_{\beta
}^{T}$ for the integral equation (\ref{intergralEq}). The weak convergence
$[u,v]\rightharpoonup$ $[u_{0},v_{0}]$ as $t\rightarrow0^{+}$ follows by
standard arguments and is left to the reader.

Now we turn to the Lipschitz continuity\ of the data-solution map. Let
$[u,v],[\tilde{u},\tilde{v}]\in$ $\mathcal{H}_{2\eta}$ be two mild solutions
with initial data $[u_{0},v_{0}],[\tilde{u}_{0},\tilde{v}_{0}],$ respectively.
Proceeding similarly to (\ref{aux3}), we obtain that
\begin{align*}
\Vert\lbrack u,v]-[\tilde{u},\tilde{v}]\Vert_{\mathcal{L}_{\beta}^{T}} &
=\Vert G(t)[u_{0}-\tilde{u}_{0},v_{0}-\tilde{v}_{0}]\Vert_{\mathcal{L}_{\beta
}^{T}}+\Vert\mathcal{G}(u)-\mathcal{G}(\tilde{u})\Vert_{\mathcal{G}_{\beta
}^{T}}\\
&  \leq C_{1}\Vert\lbrack u_{0},v_{0}]-[\tilde{u}_{0},\tilde{v}_{0}%
]\Vert_{(\frac{b+1}{b},\infty)}+K_{\beta}T^{1-b\beta}2^{b}\eta^{b-1}%
\Vert\lbrack u,v]-[\tilde{u},\tilde{v}]\Vert_{\mathcal{L}_{\beta}^{T}},
\end{align*}
which implies the desired continuity, because $K_{\beta}T^{1-b\beta}2^{b}%
\eta^{b-1}<1.$ \fin

\noindent$L^{(b+1,d)}-$\textbf{regularity.} Due to the above contraction
argument, the solution $[u,v]$ is the limit in $\mathcal{H}_{2\eta}$ of the
Picard sequence $\{[u_{k},v_{k}]\}_{k\in\mathbb{N}}$ defined as
\begin{equation}
\lbrack u_{1},v_{1}]=G(t)[u_{0},v_{0}]\text{ and }[u_{k+1},v_{k+1}%
]=G(t)[u_{0},v_{0}]+\mathcal{G}(u_{k}).\label{seq2}%
\end{equation}
We claim that the bounds%

\begin{equation}
\sup_{0<\left\vert t\right\vert <T}\left\vert t\right\vert ^{\beta}\Vert
u_{k}(\cdot,t)\Vert_{(b+1,d)}\leq C<\infty\text{ and }\sup_{0<\left\vert
t\right\vert <T}\left\vert t\right\vert ^{\beta}\Vert v_{k}(\cdot
,t)\Vert_{\Lambda^{-1}H_{(b+1,d)}^{-1}}\leq C<\infty\text{ }\label{bb1}%
\end{equation}
hold true, for all $k\in\mathbb{N}$. With them in hand and the uniqueness of
the limit in $\mathcal{S}^{\prime}(\mathbb{H}^{n})$, we obtain (\ref{reg1}).

The remainder of the proof is to show (\ref{bb1}). For that, using that
$\Vert\lbrack u_{k},v_{k}]\Vert_{\mathcal{L}_{\beta}^{T}}\leq2\eta,$ Lemma
\ref{estim1} (with $p=b+1)$ and Proposition \ref{lem8} (with $\widetilde{u}%
=0$), we can handle (\ref{seq2}) as%

\begin{align}
\sup_{0<\left\vert t\right\vert <T}\left\vert t\right\vert ^{\beta}\Vert\text{
}[u_{k+1},v_{k+1}]\Vert_{X_{(b+1,d)}}  & \leq C_{G}\Vert\lbrack u_{0}%
,v_{0}]\Vert_{(\frac{b+1}{b},d)}\nonumber\\
& +{K}_{\beta}T^{1-b\beta}\sup_{0<\left\vert t\right\vert <T}\left\vert
t\right\vert ^{\beta}\Vert u_{k}(\cdot,t)\Vert_{(b+1,d)}\left(  \sup
_{0<\left\vert t\right\vert <T}\left\vert t\right\vert ^{\beta}\Vert
u_{k}(\cdot,t)\Vert_{(b+1,\infty)}\right)  ^{b-1}\nonumber\\
& \leq C_{G}\Vert\lbrack u_{0},v_{0}]\Vert_{(\frac{b+1}{b},d)}+{K}_{\beta
}T^{1-b\beta}2^{b-1}\eta^{b-1}\sup_{0<\left\vert t\right\vert <T}\left\vert
t\right\vert ^{\beta}\Vert u_{k}(\cdot,t)\Vert_{(b+1,d)}.\label{sa2}%
\end{align}
\bigskip Next, denote $D_{1}=C_{G}\Vert\lbrack u_{0},v_{0}]\Vert_{(\frac
{b+1}{b},d)}<\infty$ and
\[
D_{k}=\sup_{0<\left\vert t\right\vert <T}\left\vert t\right\vert ^{\beta}%
\Vert\lbrack u_{k},v_{k}]\Vert_{X_{(b+1,d)}}=\sup_{0<\left\vert t\right\vert
<T}\left\vert t\right\vert ^{\beta}\max\left\{  \Vert u_{k}\Vert
_{(b+1,d)},\Vert v_{k}\Vert_{\Lambda^{-1}H_{(b+1,d)}^{-1}}\right\}  ,\text{
for }k\geq2.\text{ }%
\]
Note that estimate (\ref{sa2}) can be rewritten as $D_{k+1}\leq D_{1}%
+{K}_{\beta}T^{1-b\beta}2^{b-1}\eta^{b-1}D_{k}$. Now, choosing $T>0$ such that
${K}_{\beta}T^{1-b\beta}2^{b-1}\eta^{b-1}<1/2$, and carrying out an induction
argument, we obtain that $D_{k}\leq2D_{1}=C,$ for all $k\in\mathbb{N}$, and
then the desired claim follows.\fin

\subsection{Proof of the global well-posedness result}

\label{S44}

\textbf{Global well-posedness.} Consider the closed ball $\mathcal{H}%
_{2\varepsilon}=\{[u,v]\in\mathcal{L}_{\alpha_{1},\alpha_{2}};\Vert\lbrack
u,v]\Vert_{\mathcal{L}_{\alpha_{1},\alpha_{2}}}\leq2\varepsilon\}$ in
$\mathcal{L}_{\alpha_{1},\alpha_{2}}.$\ Recalling the map $\Phi$ defined in
(\ref{fi2-1}), it follows from Proposition \ref{lem8} (ii) with $h=0$ and
$d=\infty$ that
\begin{align}
\Vert\Phi([u,v])-\Phi([\tilde{u},\tilde{v}])\Vert_{\mathcal{L}_{\alpha
_{1},\alpha_{2}}} &  =\Vert\mathcal{G}(u)-\mathcal{G}(\tilde{u})\Vert
_{\mathcal{L}_{\alpha_{1},\alpha_{2}}}\nonumber\\
&  \leq K_{\alpha_{1},\alpha_{2},0}\Vert\lbrack u,v]-[\tilde{u},\tilde
{v}]\Vert_{\mathcal{L}_{\alpha_{1},\alpha_{2}}}\left(  \Vert\lbrack
u,v]\Vert_{\mathcal{L}_{\alpha_{1},\alpha_{2}}}^{b-1}+\Vert\lbrack\tilde
{u},\tilde{v}]\Vert_{\mathcal{L}_{\alpha_{1},\alpha_{2}}}^{b-1}\right)
\label{aux4}\\
&  \leq2^{b}\varepsilon^{b-1}K_{\alpha_{1},\alpha_{2},0}\Vert\lbrack
u,v]-[\tilde{u},\tilde{v}]\Vert_{\mathcal{L}_{\alpha_{1},\alpha_{2}}%
},\nonumber
\end{align}
for all $[u,v],[\tilde{u},\tilde{v}]\in\mathcal{H}_{2\varepsilon}.$ Since
$\mathcal{G}(0)=0$ and $\Vert G(t)[u_{0},v_{0}]\Vert_{\mathcal{L}_{\alpha
_{1},\alpha_{2}}}=\Vert\lbrack u_{0},v_{0}]\Vert_{\mathcal{E}_{0}}%
\leq\varepsilon,$ inequality (\ref{aux4}) with $[\tilde{u},\tilde{v}]=0$ yields%

\begin{align}
\Vert\Phi([u,v])\Vert_{\mathcal{L}_{\alpha_{1},\alpha_{2}}} &  \leq\Vert
G(t)[u_{0},v_{0}]\Vert_{\mathcal{L}_{\alpha_{1},\alpha_{2}}}+\Vert
\mathcal{G}(u)\Vert_{\mathcal{L}_{\alpha_{1},\alpha_{2}}}\nonumber\\
&  \leq\varepsilon+K_{\alpha_{1},\alpha_{2},0}\Vert\lbrack u,v]\Vert
_{\mathcal{L}_{\alpha_{1},\alpha_{2}}}^{b}\nonumber\\
&  \leq\varepsilon+2^{b}\varepsilon^{b}K_{\alpha_{1},\alpha_{2},0}%
\leq2\varepsilon,\label{aux6}%
\end{align}
for all $[u,v]\in\mathcal{H}_{2\varepsilon}$, provided that $2^{b}%
\varepsilon^{b-1}K_{\alpha_{1},a_{2},0}<1.$ Taking $0<\varepsilon
<(2^{b}K_{\alpha_{1},a_{2},0})^{-1/(b-1)},$ estimates (\ref{aux4}) and
(\ref{aux6}) imply that $\Phi$ is a contraction in $\mathcal{H}_{2\varepsilon
}$.  It follows that integral equation (\ref{intergralEq}) has a unique
solution $[u,v]$ satisfying $\Vert\lbrack u,v]\Vert_{\mathcal{L}_{\alpha
_{1},\alpha_{2}}}\leq2\varepsilon,$ which is given by the fixed point of
$\Phi$ in $\mathcal{H}_{2\varepsilon}$.

Next, we show the Lipschitz continuity of the data-solution map. For that, let
$[u,v],[\tilde{u},\tilde{v}]\in\mathcal{H}_{2\varepsilon}$ be two solutions
for (\ref{intergralEq}) with initial data $[u_{0},v_{0}],[\tilde{u}_{0}%
,\tilde{v}_{0}]\in\mathcal{E}_{0}$, respectively. Then, we can estimate
\begin{align*}
\Vert\lbrack u,v]-[\tilde{u},\tilde{v}]\Vert_{\mathcal{L}_{\alpha_{1}%
,\alpha_{2}}} &  =\Vert G(t)[u_{0}-\tilde{u}_{0},v_{0}-\tilde{v}_{0}%
]\Vert_{\mathcal{L}_{\alpha_{1},\alpha_{2}}}+\Vert\mathcal{G}(u)-\mathcal{G}%
(\tilde{u})\Vert_{\mathcal{L}_{\alpha_{1},\alpha_{2}}}\\
&  \leq\Vert\lbrack u_{0}-\tilde{u}_{0},v_{0}-\tilde{v}_{0}]\Vert
_{\mathcal{E}_{0}}+2^{b}\varepsilon^{b-1}K_{\alpha_{1},\alpha_{2},0}%
\Vert\lbrack u,v]-[\tilde{u},\tilde{v}]\Vert_{\mathcal{L}_{\alpha_{1}%
,\alpha_{2}}}.
\end{align*}
Reminding that $2^{b}\varepsilon^{b-1}K_{\alpha_{1},\alpha_{2},0}<1$ in the
above inequality, we are done.

$L^{(b+1,d)}$\textbf{-regularity. }Due to the contraction argument, we have
that the solution $[u,v]$ obtained in item (i) is the limit in $\mathcal{L}%
_{\alpha_{1},\alpha_{2}}$ of the Picard sequence (\ref{seq2}). Denote%
\begin{equation}
\Gamma_{k,h}=\sup_{0<\left\vert t\right\vert <1}\left\vert t\right\vert
^{\alpha_{1}+h}\left\Vert [u_{k},v_{k}]\right\Vert _{X_{(b+1,d)}}%
+\sup_{\left\vert t\right\vert >1}\left\vert t\right\vert ^{\alpha_{2}%
+h}\left\Vert [u_{k},v_{k}]\right\Vert _{X_{(b+1,d)}},\label{aux33}%
\end{equation}
where $\{[u_{k},v_{k}]\}_{k\in\mathbb{N}}$ is as in (\ref{seq2}). We are going
to show that (\ref{aux33}) is uniformly bounded w.r.t. $k\in\mathbb{N}$,
provided that $\Vert\lbrack u_{0},v_{0}]\Vert_{\mathcal{E}_{0}}\leq
\varepsilon_{0}$ for some $0<\varepsilon_{0}\leq\varepsilon$. In fact, if
$\Vert\lbrack u_{0},v_{0}]\Vert_{\mathcal{E}_{0}}\leq\varepsilon_{0}$, we have
that the solution $[u,v]\in\mathcal{H}_{2\varepsilon_{0}}$ and $\{[u_{k}%
,v_{k}]\}_{k\in\mathbb{N}}\subset\mathcal{H}_{2\varepsilon_{0}}$.

Note that $\Gamma_{1,h}$ is finite by hypothesis. Applying now Proposition
\ref{lem8} (ii) with $\tilde{u}=0$ and taking $\varepsilon_{0}$ such that
$R=2^{b-1}\varepsilon_{0}^{b-1}K_{\alpha_{1},\alpha_{2},h}<1,$ we obtain that%

\begin{align}
\Gamma_{k+1,h}= &  \sup_{0<\left\vert t\right\vert <1}\left\vert t\right\vert
^{\alpha_{1}}\left\Vert [u_{k+1},v_{k+1}]\right\Vert _{X_{(b+1,d)}}%
+\sup_{\left\vert t\right\vert >1}\left\vert t\right\vert ^{\alpha_{2}%
}\left\Vert [u_{k+1},v_{k+1}]\right\Vert _{X_{(b+1,d)}}\nonumber\\
&  \leq\sup_{0<\left\vert t\right\vert <1}\left\vert t\right\vert ^{\alpha
_{1}}\left\Vert G(t)[u_{0},v_{0}]\right\Vert _{X_{(b+1,d)}}+\sup_{\left\vert
t\right\vert >1}\left\vert t\right\vert ^{\alpha_{2}}\left\Vert G(t)[u_{0}%
,v_{0}]\right\Vert _{X_{(b+1,d)}}\nonumber\\
&  +K_{\alpha_{1},\alpha_{2},h}(\sup_{0<\left\vert t\right\vert <1}\left\vert
t\right\vert ^{\alpha_{1}+h}\Vert u_{k}\Vert_{(b+1,d)}+\sup_{\left\vert
t\right\vert >1}\left\vert t\right\vert ^{\alpha_{2}+h}\Vert u_{k}%
\Vert_{(b+1,d)})\nonumber\\
&  \times\left(  \sup_{0<\left\vert t\right\vert <1}\left\vert t\right\vert
^{\alpha_{1}}\Vert u_{k}\Vert_{(b+1,\infty)}+\sup_{\left\vert t\right\vert
>1}\left\vert t\right\vert ^{\alpha_{2}}\Vert u_{k}\Vert_{(b+1,\infty
)}\right)  ^{b-1}\nonumber\\
&  \leq\Gamma_{1,h}+2^{b-1}\varepsilon_{0}^{b-1}K_{\alpha_{1},\alpha_{2}%
,h}\Gamma_{k,h}\leq\frac{1}{1-R}\Gamma_{1,h}\text{,}\label{aux777}%
\end{align}
for all $k\in\mathbb{N}$, and then the desired uniform boundedness follows. Leveraging this boundedness along with the uniqueness of the limit in the sense of distributions, we get (\ref{reg-2-1})-(\ref{reg-2-2}). \fin

\section{Asymptotic behavior}

\label{S5}

\subsection{Scattering}

\label{S51}

In this section, we show a scattering property for the solutions derived in the preceding sections.

\begin{theorem}
[Scattering]\label{scat}Assume the same hypotheses of Theorem \ref{GWP} and
let $[u,v]$ be the mild solution of (\ref{Bou1}) with initial data $[u_{0},v_{0}]$. Then, there exist $[u_{0}^{+},v_{0}^{+}]$ and $[u_{0}^{-},v_{0}^{-}]$ belonging to $\mathcal{E}_{0}$ such that
\begin{align}
\left\Vert \lbrack u(t)-u^{+}(t),v(t)-v^{+}(t)]\right\Vert _{X_{(b+1,d)}} &
=O(t^{-b(\alpha_{2}+h)}),\text{ as }t\rightarrow+\infty,\label{scat10}\\
\left\Vert \lbrack u(t)-u^{-}(t),v(t)-v^{-}(t)]\right\Vert _{X_{(b+1,d)}} &
=O(\left\vert t\right\vert ^{-b(\alpha_{2}+h)}),\text{ as }t\rightarrow
-\infty,\label{scat102}%
\end{align}
where $[u^{+}(t),v^{+}(t)]$ and $[u^{-}(t),v^{-}(t)]$ are the unique solutions
of the linear problem (\ref{linear11}) with respective initial data
{$[u_{0}^{+},v_{0}^{+}]$ and $[u_{0}^{-},v_{0}^{-}]$.}
\end{theorem}

\textbf{Proof.} We are going to show the scattering property only for $t>0$.
The case $t\rightarrow-\infty$ follows similarly and so we omit it. Consider
the data
\begin{equation}
\lbrack u_{0}^{+},v_{0}^{+}]=[u_{0},v_{0}]-\int_{0}^{\infty}%
G(-s)[0,f(u(s))]ds\in\mathcal{E}_{0}\label{aux-scat-20}%
\end{equation}
and the corresponding solution $[u^{+},v^{+}]=G(t)[u_{0}^{+},v_{0}^{+}]$ of
the linear problem (\ref{linear11}) for $t>0.$ In view of (\ref{aux-scat-20}),
we can express $[u^{+},v^{+}]$ as
\begin{equation}
\lbrack u^{+},v^{+}]=G(t)[u_{0},v_{0}]-\int_{0}^{\infty}%
G(t-s)[0,f(u(s))]ds.\label{aux12}%
\end{equation}
As we are interested in large values {}{}of time $t$, we can assume that
$t>1$. Recalling that $u$ satisfies (\ref{intergralEq}), as well as the
property (\ref{reg-2-1}) and considering $\Gamma_{1,h},R$ as in (\ref{aux777}), we arrive
at
\begin{align*}
\Vert\lbrack u-u^{+},v-v^{+}]\Vert_{X_{(b+1,d)}} &  =\left\Vert \int
_{t}^{\infty}G(t-s)[0,f(u(s))]ds\right\Vert _{X_{(b+1,d)}}\\
&  \leq C\left(  \int_{t}^{t+1}(s-t)^{-\frac{n}{2}(\frac{b-1}{b+1}%
)}s^{-b(\alpha_{2}+h)}ds+\int_{t+1}^{\infty}(s-t)^{-\frac{3}{2}}%
s^{-b(\alpha_{2}+h)}ds\right)  \\
&  \times\left(  \sup_{s>1}s^{\alpha_{2}+h}\Vert\lbrack u(s)\Vert
_{(b+1,d)}\right)  ^{b}\\
&  \leq C(\frac{1}{1-R}\Gamma_{1,h})^{b}\left(  t^{-b(\alpha_{2}+h)}\int
_{t}^{t+1}(s-t)^{-\beta}ds+(t+1)^{-b(\alpha_{2}+h)}\int_{t+1}^{\infty
}(s-t)^{-\frac{3}{2}}ds\right)  \\
&  \leq Ct^{-b(\alpha_{2}+h)},
\end{align*}
as required.\fin

\subsection{Stability}

\label{S52} In the next result we show that the solutions obtained in Theorem
\ref{GWP} present a kind of polynomial time-weighted stability

\begin{theorem}
[Polynomial stability]\label{Stable} Suppose the same hypotheses of Theorem
\ref{GWP} but with the strict condition $\alpha_{2}<\frac{3}{2}-h$. Assume
also that $[u,v]$ and $[\tilde{u},\tilde{v}]$ belonging to $\mathcal{L}%
_{\alpha_{1},\alpha_{2}}$ are two mild solutions of (\ref{Bou1}) with respective
initial data $[u_{0},v_{0}]$ and $[\tilde{u}_{0},\tilde{v}_{0}]$. Then, we
have that
\begin{equation}
\lim_{|t|\rightarrow\infty}|t|^{\alpha_{2}+h}\left\Vert [u(t)-\tilde
{u}(t),v(t)-\tilde{v}(t)]\right\Vert _{X_{(b+1,d)}}=0\label{Stability}%
\end{equation}
if and only if
\begin{equation}
\lim_{|t|\rightarrow\infty}|t|^{\alpha_{2}+h}\left\Vert G(t)[u_{0}-\tilde
{u}_{0},v_{0}-\tilde{v}_{0}]\right\Vert _{X_{(b+1,d)}}=0.\label{condition}%
\end{equation}
In view of (\ref{semigroup}), condition (\ref{condition}) holds particularly
for $\varphi=u_{0}-\tilde{u}_{0}$ and $\psi=v_{0}-\tilde{v}_{0}$ belonging to
$L^{(\frac{b+1}{b},d)}(\mathbb{H}^{n})$.
\end{theorem}

\textbf{Proof.} First we show (\ref{Stability}) by assuming (\ref{condition}). Again we treat only the case $t\rightarrow\infty$ and leave
the opposite case to the reader. Taking the difference between the two mild
solutions and afterwards computing the norm $t^{\alpha_{2}+h}\left\Vert
\cdot\right\Vert _{X_{(b+1,d)}}$ on the resulting expression, we arrive at%
\begin{align}
t^{\alpha_{2}+h}\left\Vert [u(t)-\tilde{u}(t),v(t)-\tilde{v}(t)]\right\Vert
_{X_{(b+1,d)}} &  \leq t^{\alpha_{2}+h}\left\Vert G(t)[u_{0}-\tilde{u}%
_{0},v_{0}-\tilde{v}_{0})]\right\Vert _{X_{(b+1,d)}}\nonumber\\
&  +|t|^{\alpha_{2}+h}\int_{0}^{t}\left\Vert G(t-s)[0,f(u(s))-f(\tilde
{u}(s))]ds\right\Vert _{X_{(b+1,d)}}ds.\label{ine-13}%
\end{align}
Now consider $\varepsilon>0$ in such a way that $R=2^{b-1}\varepsilon
^{b-1}K_{\alpha_{1},\alpha_{2},h}<1$, where $K_{\alpha_{1},\alpha_{2},h}$ is
the constant that appears in the corresponding estimates in Proposition
\ref{lem8}. For $\left\Vert [u_{0},v_{0}]\right\Vert _{\mathcal{E}_{0}}%
\leq\varepsilon$ and $\left\Vert [\tilde{u}_{0},\tilde{v}_{0}]\right\Vert
_{\mathcal{E}_{0}}\leq\varepsilon,$ Theorem \ref{GWP} (i) gives that
$\left\Vert [u,v]\right\Vert _{\mathcal{L}_{\alpha_{1},\alpha_{2}}}%
\leq2\varepsilon$ and$\,\left\Vert [\tilde{u},\tilde{v}]\right\Vert
_{\mathcal{L}_{\alpha_{1},\alpha_{2}}}\leq2\varepsilon,$. Moreover, it follows
from the proof of item (ii) of the same theorem that (see (\ref{aux777}))
\begin{align}
\sup_{|t|<1}|t|^{\alpha_{1}+h}\left\Vert [u(t),v(t)]\right\Vert _{X_{(b+1,d)}%
}+\sup_{|t|>1}|t|^{\alpha_{2}+h}\left\Vert [u(t),v(t)]\right\Vert
_{X_{(b+1,d)}} &  \leq\frac{1}{1-R}\Gamma_{1,h},\label{aux-bound-1}\\
\sup_{|t|<1}|t|^{\alpha_{1}+h}\left\Vert [\tilde{u}(t),\tilde{v}%
(t)]\right\Vert _{X_{(b+1,d)}}+\sup_{|t|>1}|t|^{\alpha_{2}+h}\left\Vert
[\tilde{u}(t),\tilde{v}(t)]\right\Vert _{X_{(b+1,d)}} &  \leq\frac{1}%
{1-R}\Gamma_{1,h}.\label{aux-bound-2}%
\end{align}
For $t\geq2$, with the above estimates in hand, the second term in the R.H.S.
of (\ref{ine-13}) can be handled along the following lines
\begin{align}
&  t^{\alpha_{2}+h}\int_{0}^{t}\left\Vert G(t-s)[0,f(u(s))-f(\tilde
{u}(s)]ds\right\Vert _{X_{(b+1,d)}}ds\nonumber\\
&  \leq t^{\alpha_{2}+h}\left(  C\int_{0}^{1}\phi_{b+1}(t-s)s^{-\alpha_{1}%
b-h}s^{\alpha_{1}+h}\left\Vert [u(s)-\tilde{u}(s),v(s)-\tilde{v}%
(s)]\right\Vert _{X_{(b+1,d)}}ds\right.  \nonumber\\
\text{ \ } &  \left.  \text{ \ \ \ \ \ \ \ \ \ \ }+C\int_{1}^{t}\phi
_{b+1}(t-s)s^{-\alpha_{2}b-h}s^{\alpha_{2}+h}\left\Vert [u(s)-\tilde
{u}(s),v(s)-\tilde{v}(s)]\right\Vert _{X_{(b+1,d)}}ds\right)  \nonumber\\
&  \text{ \ \ \ \ \ \ \ \ \ \ \ \ }\times\left(  \left\Vert \lbrack
u,v]\right\Vert _{\mathcal{L}_{\alpha_{1},\alpha_{2}}}^{b-1}+\left\Vert
[\tilde{u},\tilde{v}]\right\Vert _{\mathcal{L}_{\alpha_{1},\alpha_{2}}}%
^{b-1}\right)  \\
&  \leq C2^{b+1}\varepsilon^{b-1}\frac{\Gamma_{1,h}}{1-R}\left(  t^{\alpha
_{2}+h}\int_{0}^{1}(t-s)^{-\frac{3}{2}}s^{-\alpha_{1}b-h}ds+t^{\alpha_{2}%
+h}\int_{1}^{t/2}(t-s)^{-\frac{3}{2}}s^{-\alpha_{2}b-h}ds\right.  \nonumber\\
\hspace{1cm} &  \left.  \text{
\ \ \ \ \ \ \ \ \ \ \ \ \ \ \ \ \ \ \ \ \ \ \ \ \ \ }+t^{\alpha_{2}+h}%
\int_{t/2}^{t-1}(t-s)^{-\frac{3}{2}}s^{-\alpha_{2}b-h}ds+t^{\alpha_{2}+h}%
\int_{t-1}^{t}(t-s)^{-\beta}s^{-\alpha_{2}b-h}ds\right)  ,\label{aux-asymp-11}%
\end{align}
where we have split the integral $\int_{1}^{t}...ds$ in three parts and used
\eqref{aux-bound-1} and \eqref{aux-bound-2}. Next it follows that
\begin{align}
&  \text{R.H.S. of (\ref{aux-asymp-11})}\nonumber\\
&  \leq C2^{b+1}\varepsilon^{b-1}\frac{\Gamma_{1,h}}{1-R}\left(  t^{\alpha
_{2}+h}(t-1)^{-\frac{3}{2}}\int_{0}^{1}s^{-\alpha_{1}b-h}ds+t^{\alpha_{2}%
+h}(t/2)^{-\frac{3}{2}}\int_{1}^{t/2}s^{-\alpha_{2}b-h}ds\right.  \nonumber\\
&  \left.  +t^{\alpha_{2}+h}(t/2)^{-\alpha_{2}b-h}\int_{t/2}^{t-1}%
(t-s)^{-\frac{3}{2}}ds+t^{\alpha_{2}+h}(t-1)^{-\alpha_{2}b-h}\int_{t-1}%
^{t}(t-s)^{-\beta}ds\right)  \nonumber\\
&  \leq C2^{b+1}\varepsilon^{b-1}\frac{\Gamma_{1,h}}{1-R}\left(  t^{\alpha
_{2}+h-\frac{3}{2}}+t^{-\alpha_{2}(b-1)-\frac{1}{2}}+t^{-\alpha_{2}%
(b-1)}\right)  .\label{aux-asymp-12}%
\end{align}

Taking $\limsup\limits_{t\rightarrow\infty}$ in (\ref{ine-13}) and considering
estimates (\ref{aux-asymp-11})-(\ref{aux-asymp-12}) and condition
(\ref{condition}) lead us to
\begin{align*}
\limsup_{t\rightarrow\infty}t^{\alpha_{2}+h}\left\Vert [u(t)-\tilde
{u}(t),v(t)-\tilde{v}(t)]\right\Vert _{X_{(b+1,d)}} &  \leq\lim_{t\rightarrow
\infty}t^{\alpha_{2}+h}\left\Vert G(t)[u_{0}-\tilde{u}_{0},v_{0}-\tilde{v}%
_{0}]\right\Vert _{X_{(b+1,d)}}\\
&  +C2^{b+1}\varepsilon^{b-1}\frac{\Gamma_{1,h}}{1-R}\lim_{t\rightarrow\infty
}(t^{\alpha_{2}+h-\frac{3}{2}}+t^{-\alpha_{2}(b-1)-\frac{1}{2}}+t^{-\alpha
_{2}(b-1)})\\
&  =0,
\end{align*}
for $0\leq\alpha_{2}<\frac{3}{2}-h$, which implies (\ref{Stability}).

In the sequel, we address the reciprocal statement. Assuming (\ref{Stability})
and using the same estimates  (\ref{aux-asymp-11})-(\ref{aux-asymp-12}) for
\[
t^{\alpha_{2}+h}\int_{0}^{t}\left\Vert G(t-s)[0,f(u(s))-f(\tilde
{u}(s))]ds\right\Vert _{X_{(b+1,d)}}ds,
\]
it follows that%
\begin{align*}
&  \limsup_{t\rightarrow+\infty}t^{\alpha_{2}+h}\left\Vert G(t)[u_{0}%
-\tilde{u}_{0},v_{0}-\tilde{v}_{0}]\right\Vert _{X_{(b+1,d)}}\\
&  \leq\limsup_{t\rightarrow+\infty}t^{\alpha_{2}+h}\left\Vert [u(t)-\tilde
{u}(t),v(t)-\tilde{v}(t)]\right\Vert _{X_{(b+1,d)}}\\
&  +\limsup_{t\rightarrow+\infty}t^{\alpha_{2}+h}\int_{0}^{t}\left\Vert
G(t-s)[0,f(u(s))-f(\tilde{u}(s))]\right\Vert _{X_{(b+1,d)}}ds\\
&  =0+C2^{b+1}\varepsilon^{b-1}\frac{\Gamma_{1,h}}{1-R}\lim_{t\rightarrow
+\infty}(t^{\alpha_{2}+h-\frac{3}{2}}+t^{-\alpha_{2}(b-1)-\frac{1}{2}%
}+t^{-\alpha_{2}(b-1)})=0,
\end{align*}
which concludes the proof.\fin

\begin{remark}
\label{rem} Theorem \ref{Stable} allows us to improve the scattering decay in
Theorem \ref{scat}. In fact, for $[u^{\pm},v^{\pm}]$ as in Theorem \ref{scat},
we can show that
\begin{equation}
\left\Vert \lbrack u(t)-u^{\pm}(t),v(t)-v^{\pm}(t)]\right\Vert _{X_{(b+1,d)}%
}=o(|t|^{-b(\alpha_{2}+h)}),\text{ as }t\rightarrow\pm\infty,\label{improve2}%
\end{equation}
provided that
\begin{equation}
\lim_{t\rightarrow\pm\infty}|t|^{\alpha_{2}+h}\left\Vert G(t)[u_{0}%
,v_{0}]\right\Vert _{X_{(b+1,d)}}=0.\label{improve1}%
\end{equation}
In other words, we have a gain by being able to replace $O$ with $o$ in
(\ref{scat10})-(\ref{scat102}). For that matter, using (\ref{improve1}) in
Theorem \ref{Stable} with $[\tilde{u}_{0},\tilde{v}_{0}]=0$ and $[\tilde
{u},\tilde{v}]=0,$ we obtain that
\begin{equation}
\lim_{t\rightarrow\pm\infty}|t|^{\alpha_{2}+h}\left\Vert
[u(t),v(t)]\right\Vert _{X_{(b+1,d)}}=0.\label{improve3}%
\end{equation}
Then, for $t>k\geq1,$ we can estimate
\begin{align}
&  \left\Vert [u(t)-u^{+}(t),v(t)-v^{+}(t)]\right\Vert _{X_{(b+1,d)}%
}=\left\Vert \int_{t}^{\infty}G(t-s)[0,f(u(s))]ds\right\Vert _{X_{(b+1,d)}%
}\nonumber\\
&  \leq C\int_{t}^{\infty}\phi_{b+1}(t-s)\left\Vert [u(s),v(s)]\right\Vert
_{X_{(b+1,d)}}^{b}ds\nonumber\\
&  \leq C\left(  \int_{t}^{t+1}(s-t)^{-\beta}s^{-b(\alpha_{2}+h)}ds+\int
_{t+1}^{\infty}(s-t)^{-\frac{3}{2}}s^{-b(\alpha_{2}+h)}ds\right)  \nonumber\\
&  \times\left(  \sup_{s>1}s^{\alpha_{2}+h}\left\Vert [u(s),v(s)]\right\Vert
_{X_{(b+1,\infty)}}\right)  ^{b-1}\sup_{s>k}s^{\alpha_{2}+h}\left\Vert
[u(s),v(s)]\right\Vert _{X_{(b+1,d)}}\nonumber\\
&  \leq C\frac{1}{(1-R)^{b-1}}\Gamma_{1,h}^{b-1}t^{-b(\alpha_{2}+h)}\sup
_{t>k}t^{\alpha_{2}+h}\left\Vert [u(t),v(t)]\right\Vert _{X_{(b+1,d)}.%
}\label{aux-scat-10}%
\end{align}
Now, multiplying both sides of (\ref{aux-scat-10}) by $t^{b(\alpha_{2}+h)}$
and afterwards taking the supremum over $t>k,$ we arrive at%
\[
\sup_{t>k}t^{b(\alpha_{2}+h)}\left\Vert [u(t)-u^{+}(t),v(t)-v^{+}%
(t)]\right\Vert _{X_{(b+1,d)}}\leq C\frac{1}{(1-R)^{b-1}}\Gamma_{1,h}%
^{b-1}\sup_{t>k}t^{\alpha_{2}+h}\left\Vert [u(t),v(t)]\right\Vert
_{X_{(b+1,d)}},
\]
which, making $k\rightarrow\infty$, leads us to
\begin{align*}
0\leq\limsup_{t\rightarrow\infty}t^{b(\alpha_{2}+h)}\left\Vert [u(t)-u^{+}%
(t),v(t)-v^{+}(t)]\right\Vert _{X_{(b+1,d)}} &  \leq C\frac{1}{(1-R)^{b-1}%
}\Gamma_{1,h}^{b-1}\limsup_{t\rightarrow\infty}t^{\alpha_{2}+h}\left\Vert
[u(t),v(t)]\right\Vert _{X_{(b+1,d)}}\\
&  =0,
\end{align*}
because the limit (\ref{improve3}), and we are done.
\end{remark}

\section*{Appendix}

In this part we provide the proof of inequality \eqref{Osci2} in detail. We
first recall this inequality: Let $\widetilde{m}:\mathbb{R}\rightarrow
\mathbb{C}$ be a smooth function supported in $[-2,\,-1/2]\cup\lbrack1/2,\,2]$
and assume that the phase $-t\psi(\lambda)\pm r\lambda$ has no critical points
in the support of $\widetilde{m}$. Then, we have that
\[
\left\vert \int_{\mathbb{R}}e^{i\left(  -t\psi(\lambda)\pm r\lambda\right)
}\widetilde{m}(\lambda)d\lambda\right\vert \leq\frac{c}{(1+D)^{k}}\sum
_{j=0}^{k}\int_{\mathbb{R}}|\partial_{\lambda}^{j}\widetilde{m}(\lambda
)|d\lambda,
\]
where $\psi(\lambda)=\sqrt{(\lambda^{2}+\rho^{2})(\lambda^{2}+\rho^{1}+1)}$
for $\rho=\frac{n-1}{2}$ and
\[
D=\min\limits_{|\lambda|\in\lbrack1/2,\,2]}\left\vert -t\partial_{\lambda}%
\psi\right\vert +r.
\]

Since $-t\psi(\lambda)\pm r\lambda$ has no critical points in the support of
$\widetilde{m}$, we have that $-t\partial_{\lambda}\psi\pm r\neq0$ for
$\lambda\in K=$ supp $(\widetilde{m}(\lambda))=[-2,-\frac{1}{2}]\cup
\lbrack\frac{1}{2},2]$. Combining with the fact that $\psi(\lambda
)\simeq\lambda^{2}$ and $\partial_{\lambda}\psi\simeq2\lambda$, we arrive at
\[
C_{1}(\left\vert t\partial_{\lambda}\psi\right\vert +r)\leq\left\vert
-t\partial_{\lambda}\psi\pm r\right\vert \leq4(\left\vert t\partial_{\lambda
}\psi\right\vert +r),\text{ for all }\lambda\in K,
\]
where $C_{1}$ is independent of $t,r,\lambda.$

Using the fact that
\[
\frac{1}{[i(-t\partial_{\lambda}\psi \pm r)]}\partial_{\lambda}e^{i(-t\partial
_{\lambda}\psi\pm r\lambda)}=e^{i(-t\psi(\lambda)\pm r\lambda)},
\]
it follows that $(k=1)$%
\begin{align*}
\left\vert \int_{\mathbb{R}}e^{i(-t\psi(\lambda)\pm r\lambda)}\widetilde
{m}(\lambda)d\lambda\right\vert  &  =\left\vert \int_{\mathbb{R}}%
\partial_{\lambda}e^{i(-t\psi(\lambda)\pm r\lambda)}\frac{1}{[i(-t\partial_{\lambda}\psi \pm r)]}\widetilde{m}(\lambda)d\lambda\right\vert \\
&  =\left\vert \int_{\mathbb{R}}e^{i(-t\psi(\lambda)\pm r\lambda)}%
\partial_{\lambda}\left(  \frac{1}{[i(-t\partial_{\lambda}\psi \pm r)]}%
\widetilde{m}(\lambda)\right)  d\lambda\right\vert \\
&  =\left\vert \int_{K}e^{i(-t\psi(\lambda)\pm r\lambda)}\partial_{\lambda
}\left(  \frac{1}{[i(-t\partial_{\lambda}\psi \pm r)]}\widetilde{m}%
(\lambda)\right)  d\lambda\right\vert \\
&  \leq\int_{K}\left\vert \partial_{\lambda}\left(  \frac{1}{[i(-t\partial_{\lambda}\psi \pm r)]}\widetilde{m}(\lambda)\right)  \right\vert d\lambda\\
&  =\int_{K}\left\vert \partial_{\lambda}\left(  \frac{1}{[i(-t\partial_{\lambda}\psi \pm r)]}\right)  \widetilde{m}(\lambda)+\frac{1}%
{[i(-t\partial_{\lambda}\psi \pm r)]}\partial_{\lambda}\widetilde{m}%
(\lambda)\right\vert d\lambda\\
&  \leq C\frac{1}{\min\limits_{\lambda\in K}\left\vert -t\partial_{\lambda
}\psi \pm r\right\vert }\left(  \int_{K}|\widetilde{m}(\lambda)|d\lambda+\int
_{K}|\partial_{\lambda}\widetilde{m}(\lambda)|d\lambda\right)  \\
&  \lesssim C\frac{1}{1+\min\limits_{\lambda\in K}\left\vert -t\partial
_{\lambda}\psi \pm r\right\vert }\left(  \int_{K}|\widetilde{m}(\lambda
)|d\lambda+\int_{K}|\partial_{\lambda}\widetilde{m}(\lambda)|d\lambda\right)
\\
&  \hbox{   (because}\min\limits_{\lambda\in K}\left\vert -t\partial_{\lambda
}\psi \pm r\right\vert \neq0\hbox{   and bounded)}\\
&  \lesssim C\frac{1}{1+\min\limits_{\lambda\in K}\left\vert -t\partial
_{\lambda}\psi\right\vert +r}\left(  \int_{K}|\widetilde{m}(\lambda
)|d\lambda+\int_{K}|\partial_{\lambda}\widetilde{m}(\lambda)|d\lambda\right)
.
\end{align*}
Repeating the same process $k$ times, we obtain
\begin{align*}
\left\vert \int_{\mathbb{R}}e^{i(-t\psi(\lambda)\pm r\lambda)}\widetilde
{m}(\lambda)d\lambda\right\vert  &  =\left\vert \int_{\mathbb{R}}%
e^{i(-t\psi(\lambda)\pm r\lambda)}\underset{d-times}{\underbrace
{\partial_{\lambda}\frac{1}{[i(-t\partial_{\lambda}\psi \pm r)]}...\partial
_{\lambda}\frac{1}{[i(-t\partial_{\lambda}\psi \pm r)]}}}\widetilde{m}%
(\lambda)d\lambda\right\vert \\
&  \leq\left\vert \int_{K}e^{i(-t\psi(\lambda)\pm r\lambda)}\underset
{k-times}{\underbrace{\partial_{\lambda}\frac{1}{[i(-t\partial_{\lambda}\psi \pm r)]}...\partial_{\lambda}\frac{1}{[i(-t\partial_{\lambda}\psi \pm r)]}}%
}\widetilde{m}(\lambda)d\lambda\right\vert \\
&  \leq\int_{K}\left\vert \underset{k-times}{\underbrace{\partial_{\lambda
}\frac{1}{[i(-t\partial_{\lambda}\psi \pm r)]}...\partial_{\lambda}\frac
{1}{[i(-t\partial_{\lambda}\psi \pm r)]}}}\widetilde{m}(\lambda)\right\vert
d\lambda\\
&  \leq C\frac{1}{(\min\limits_{\lambda\in K}\left\vert -t\partial_{\lambda}\psi \pm r \right\vert )^{k}} \int_{K}\sum_{j=0}^{k}|\partial_{\lambda}%
^{j}\widetilde{m}(\lambda)|d\lambda\\
&  \lesssim C\frac{1}{(1+\min\limits_{\lambda\in K}|-t\partial_{\lambda}%
\psi|+r)^{k}}\int_{K}\sum_{j=0}^{k}|\partial_{\lambda}^{j}%
\widetilde{m}(\lambda)|d\lambda,
\end{align*}
and we are done.

\noindent\\
\textbf{Conflict of Interest Statement:} The authors declare that there are no conflicts of interest to disclose.

\end{document}